\documentclass{amsart}
\usepackage{bm}

\usepackage{amsmath}
\usepackage{amsthm}
\usepackage{amssymb}
\usepackage{enumitem}

\usepackage{mathrsfs}

\newcommand{\R}{\mathbb{R}}
\newcommand{\C}{\mathbb{C}}
\newcommand{\N}{\mathbb{N}}
\newcommand{\lr}[1]{\langle #1 \rangle}
\newcommand{\eps}{\varepsilon}
\newcommand{\norm}[1]{\| #1 \|}
\newcommand{\Norm}[1]{\left\| #1\right\|}

\newtheorem{thm}{Theorem}[section]
\newtheorem{prop}[thm]{Proposition}
\newtheorem{dfn}[thm]{Definition}
\newtheorem{lem}[thm]{Lemma}
\newtheorem{cor}[thm]{Corollary}

\theoremstyle{remark}
\newtheorem{rmk}[thm]{Remark}

\DeclareMathOperator{\med}{med}
\DeclareMathOperator{\supp}{supp}
\DeclareMathOperator{\Id}{Id}

\allowdisplaybreaks[1]

\title[Random data Cauchy problem for NLS with derivative nonlinearity]{Random data Cauchy problem for the nonlinear Schr\"{o}dinger equation with derivative nonlinearity}

\author[H. Hirayama]{Hiroyuki Hirayama}
\address{Graduate School of Mathematics, Nagoya University, Chikusa-ku, Nagoya, 464-8602, Japan}
\email{m08035f@math.nagoya-u.ac.jp}

\author[M. Okamoto]{Mamoru Okamoto}
\address{Department of Mathematics, Institute of Engineering, Academic Assembly, Shinshu University, 4-17-1 Wakasato, Nagano City 380-8553, Japan}
\email{m\_okamoto@shinshu-u.ac.jp}
\date{}

\subjclass[2010]{35Q55}

\numberwithin{equation}{section}

\begin{document}
%\pagewiselinenumbers

\begin{abstract}
We consider the Cauchy problem for the nonlinear Schr\"{o}dinger equation with derivative nonlinearity $(i\partial _t + \Delta ) u= \pm \partial (\overline{u}^m)$ on $\R ^d$, $d \ge 1$, with random initial data, where $\partial$ is a first order derivative with respect to the spatial variable, for example a linear combination of  $\frac{\partial}{\partial x_1} , \, \dots , \, \frac{\partial}{\partial x_d}$ or $|\nabla |= \mathcal{F}^{-1}[|\xi | \mathcal{F}]$.
We prove that almost sure local in time well-posedness, small data global in time well-posedness and scattering hold in $H^s(\R ^d)$ with $s> \max \left( \frac{d-1}{d} s_c , \frac{s_c}{2}, s_c - \frac{d}{2(d+1)} \right)$ for $d+m \ge 5$, where $s$ is below the scaling critical regularity $s_c := \frac{d}{2}-\frac{1}{m-1}$.
\end{abstract}

\maketitle

\section{Introduction}

We consider the Cauchy problem for the nonlinear Schr\"{o}dinger equation with derivative nonlinearity:
\begin{equation}
\label{NLS}
\left\{
\begin{aligned}
& (i \partial _t + \Delta ) u = \pm \partial ( \overline{u}^m), \\
& u(0, \cdot ) = \phi .
\end{aligned}
\right.
\end{equation}
Here, $m$ is a positive integer, $u : \R \times \R ^d \rightarrow \C$ is an unknown function, $\phi : \R ^d \rightarrow \C$ is a given function, $\partial$ is a first order derivative with respect to the spatial variable, for example a linear combination of  $\frac{\partial}{\partial x_1} , \, \dots , \, \frac{\partial}{\partial x_d}$ or $|\nabla |= \mathcal{F}^{-1}[|\xi | \mathcal{F}]$.

The nonlinear Schr\"{o}dinger equation in \eqref{NLS} is invariant under the following transformation:
\[
u(t,x) \mapsto u_{\mu} (t,x) := \mu ^{-\frac{1}{m-1}} u \Big( \frac{t}{\mu ^2} , \frac{x}{\mu} \Big)
\]
for $\mu >0$.
A simple calculation shows
\begin{equation*}
\norm{u_{\mu}(0,\cdot )}_{\dot{H}^s} = \mu ^{-s + \frac{d}{2}-\frac{1}{m-1}} \norm{u(0,\cdot )}_{\dot{H}^s},
\end{equation*}
which implies that $s_c :=\frac{d}{2}-\frac{1}{m-1}$ is the scaling critical regularity.

We mention the previous and related results for \eqref{NLS}.
Gr\"unrock \cite{Gru} proved local in time well-posedness of \eqref{NLS} in $L^{2} (\R )$ when $m=2$ and in $H^{s}(\R ^d)$ for $s>s_{c}$ when $d\geq 1$, $d+m \geq 4$. 
The first author \cite{H14}, \cite{H13} proved that \eqref{NLS} is small data global well-posedness and scattering for $s \ge s_c$ 
if $m+d\ge 4$.
Well-posedness of the Cauchy problem for \eqref{NLS} in $d=1$ whose $\partial (\overline{u}^m)$ is replaced by $\partial _x (|u|^2u)$ is intensively studied by many authors (see, for example, \cite{HO92}, \cite{Hay93}, \cite{Tak99}, \cite{Tak01}, \cite{CKSTT01}, \cite{BL01}, \cite{CKSTT02}, \cite{Her06}, \cite{Wu}, \cite{MO} and references therein).
Presence of derivative causes some difficulties.
In the present paper, we impose that the nonlinear part of \eqref{NLS} has special structure which cancels out the worst interaction.
Owing to this property, we can recover one derivative.

The above results are deterministic results.
We consider well-posedness of \eqref{NLS} with randomized initial data.
Following the papers \cite{BOP1}, \cite{BOP2}, we define the randomization.
Let $\psi \in \mathcal{S}(\R ^d)$ satisfy
\[
\supp \psi \subset [-1,1]^d , \quad \sum _{n \in \mathbb{Z}^d} \psi (\xi -n) =1 \quad \text{for any $\xi \in \R ^d$}.
\]
Let $\{ g_n \}$ be a sequence of independent mean zero complex valued random variables on a probability space $(\Omega , \mathcal{F} ,P)$, where the real and imaginary parts of $g_n$ are independent and endowed with probability distributions $\mu _n^{(1)}$ and $\mu _n^{(2)}$.
Throughout this paper, we assume that there exists $c>0$ such that
\[
\left| \int _{\R} e^{\kappa x} d \mu _{n}^{(j)} (x) \right| \le e^{c \kappa ^2}
\]
for all $\kappa \in \R$, $n \in \mathbb{Z}^d$, $j=1,2$.
This condition is satisfied by the standard complex valued Gaussian random variables and the standard Bernoulli random variables.
We then define the Wiener randomization of $\phi$ by
\begin{equation} \label{eq:rand}
\phi ^{\omega} := \sum _{n \in \mathbb{Z}^d} g_n (\omega ) \psi (D-n) \phi .
\end{equation}

The randomization has no smoothing in terms of differentiability (\cite[Appendix B]{BT1}).
However, it improves the integrability (see for example Lemma 2.3 in \cite{BOP1}).
>From this point of view, the randomization makes the problem subcritical in some sense.
In the present paper, we focus on the case where the regularity is less than $s_c = \frac{d}{2}-\frac{1}{m-1}$ because well-posedness in $H^s(\R ^d)$ with $s \ge s_c$ holds in the deterministic setting.

\begin{thm} \label{thm:LWP}
Assume $d \ge 1$, $m \ge 2$, and $d+m \ge 5$.
Let $\max \left( \frac{d-1}{d} s_c, \frac{s_c}{2}, s_c - \frac{d}{2(d+1)} \right)<s< s_c$.
Given $\phi \in H^s(\R ^d)$, let $\phi ^{\omega}$ be its randomization defined by \eqref{eq:rand}.
Then, for almost all $\omega \in \Omega$, there exist $T_{\omega} >0$ and a unique solution $u$ to \eqref{NLS} with $u(0,x) = \phi ^{\omega}(x)$ in a space continuously embedded in
\[
S(t) \phi ^{\omega} + C((-T_{\omega},T_{\omega}) ; H^{s_c}(\R ^d)) \subset C((-T_{\omega},T_{\omega}) ; H^{s}(\R ^d)).
\]
More precisely, there exist $C, c>0$, $\gamma >0$ such that for each $0<T<1$, there exists $\Omega _{T} \subset \Omega$ with $P(\Omega _T) \ge 1- C \exp \left( - \frac{c}{T^{\gamma} \norm{\phi}_{H^s}^2} \right)$.
\end{thm}

Theorem \ref{thm:LWP} says almost sure local in time well-posedness for \eqref{NLS}.
Namely, \eqref{NLS} possesses local strong solutions for a large class of functions in $H^s(\R ^d)$ with $s<s_c$.

We find a solution $u$ which is a perturbation of $e^{it\Delta} \phi ^{\omega}$.
The linear evolution for the randomized initial data has better integrability than that for the initial data (see Lemma \ref{lem:stcStr} below), but it remains $C((-T_{\omega}, T_{\omega}) ;H^s(\R ^d))$.
On the other hand, from the smoothing effect of the linear evolution and absence of resonance interaction, the difference $u - e^{it \Delta} \phi ^{\omega}$ belongs to $C((-T_{\omega}, T_{\omega}) ; H^{s_c}(\R ^d))$ even if $\phi \in H^s(\R )$ with $s<s_c$.

\begin{rmk}
The lower bound is equivalent to
\[
\max \left( \frac{d-1}{d} s_c, \frac{s_c}{2}, s_c - \frac{d}{2(d+1)} \right)
=
\begin{cases}
\frac{s_c}{2}, & \text{if } d=1, \\
\frac{d-1}{d} s_c, & \text{if } d\ge 2\ \text{and}\ m=2, 3, \\
s_c - \frac{d}{2(d+1)}, & \text{if } d\ge 2\ \text{and}\ m\ge 4.
\end{cases}
\]
\end{rmk}

Next, we focus on global existence of the solution with small initial data.

\begin{thm} \label{thm:GWP}
Assume $d \ge 1$, $m \ge 2$, and $d+m \ge 5$.
Given $\phi \in H^s(\R ^d)$, let $\phi^{\omega}$ be its randomization defined by \eqref{eq:rand}.
Then, for almost all $\omega \in \Omega$, there exists  $\eps (\omega ) >0$ such that for every $\eps \in (0, \eps (\omega ))$, there exists a global in time solution $u$ to \eqref{NLS} with $u(0,x) = \eps \phi ^{\omega}(x)$ in a space continuously embedded in $C(\R ; H^{s}(\R ^d))$.
Moreover, the solution is scattering in the following sense:
there exists $v_{\pm}^{\omega} \in H^{s_c}(\R ^d)$ such that
\[
\norm{u(t) - S(t) (\phi ^{\omega} + v_{\pm}^{\omega})}_{H^{s_c}} \rightarrow 0
\]
as $t \rightarrow \pm \infty$.
\end{thm}

The uniqueness holds in the space $Y^s$ defined by Definition \ref{def:fs} below, 
which is a subspace continuously embedded in $S(t) \phi ^{\omega} + C(\R ; H^{s_c}(\R ^d))$.

\begin{rmk}
Theorem \ref{thm:GWP} is a consequence of the following:
there exist $C, c>0$ and $\Omega _{\phi} \subset \Omega$ such that with the following properties:
\begin{enumerate}[label=(\alph*)]
\item $P(\Omega _{\phi}) \ge 1- C \exp \left( - \frac{c}{\norm{\phi}_{H^s}^2} \right)$.
\item For each $\omega \in \Omega$, there exists a (unique) global in time solution $u$ to \eqref{NLS} with $u(0,x) = \phi ^{\omega} (x)$ in the class
\[
S(t) \phi ^{\omega} + C(\R ; H^{s_c}(\R ^d)) \subset C(\R ; H^{s}(\R ^d)).
\]
\item For each $\omega \in \Omega _{\phi}$, there exists $v_{\pm}^{\omega} \in H^{s_c}(\R ^d)$ such that
\[
\norm{u(t) - S(t) (\phi ^{\omega} + v_{\pm}^{\omega})}_{H^{s_c}} \rightarrow 0
\]
as $t \rightarrow \pm \infty$.
\end{enumerate}
\end{rmk}

The nonlinear part of \eqref{NLS} excludes the resonance, which is the worst interaction.
In other words, if an output of the nonlinear interaction is on the characteristic curve, then the at least one of the inputs is off its curve (see \eqref{mod} below).
Therefore, by using the modulation estimate and the Fourier restriction norm, we can recover one derivative.
These are also useful in the randomized initial data setting.

The number $\alpha (d,m) := \max \left( \frac{d-1}{d} s_c, \frac{s_c}{2}, s_c - \frac{d}{2(d+1)} \right)$ satisfies
\begin{align*}
& \frac{1}{s_c} \times \alpha (d,m) = \max \left( 1-\frac{1}{d}, \frac{1}{2}, 1-\frac{(m-1)d}{(d+1) ((m-1)d-2)} \right)  \rightarrow 1, \\
& s_c - \alpha (d,m) = \min \left( \frac{(m-1)d-2}{2(m-1)d}, \frac{(m-1)d-2}{4(m-1)}, \frac{d}{2(d+1)} \right) \rightarrow \frac{1}{2}
\end{align*}
as $d \rightarrow \infty$.
On the other hand, B\'{e}nyi, Oh, and Pocovnicu \cite{BOP2} showed that the cubic nonlinear Schr\"{o}dinger equation without derivative is almost sure well-posed in $H^s(\R ^d)$ with $s> \frac{d-1}{d+1} \cdot \frac{d-2}{2}$ and $d \ge 3$, where $\frac{d-2}{2}$ is the scaling critical regularity.
Here, we note that
\[
\frac{2}{d-2} \times \frac{d-1}{d+1} \cdot \frac{d-2}{2} = \frac{d-1}{d+1} \rightarrow 1, \quad
\frac{d-2}{2} - \frac{d-1}{d+1} \cdot \frac{d-2}{2} = \frac{d-2}{d+1} \rightarrow 1 .
\]
This difference comes from the fact that we rely on not only the bilinear refinement of the Strichartz estimates but also the modulation bound.

We obtain the almost sure well-posedness in $d \ge 2$ if $m \ge 3$, although the result of B\'{e}nyi et. al. is required $d \ge 3$.
One reason for this is that the scaling critical regularity of \eqref{NLS} is bigger than that of the cubic nonlinear Schr\"{o}dinger equation without derivative.
More precisely, the scaling critical regularity is zero if $d=1$, $m=3$ in our case, while the scaling critical regularity is zero if $d=2$ in the cubic nonlinear Schr\"{o}dinger equation without derivative.
Indeed, since the randomization does not improve regularity, we can not expect that almost sure well-posedness holds in the Sobolev space with negative regularity.
>From the same reason, we need the condition $d+m \ge 5$ in Theorems \ref{thm:LWP} and \ref{thm:GWP}.

Put $\mathcal{N}_{m} (u) = \partial (\overline{u}^{m})$.
Let $z(t) = z^{\omega} (t) := S(t) \phi ^{\omega}$ and $v(t) = u(t) - z(t)$ be the linear and nonlinear parts of $u$ respectively.
As in \cite{BOP2}, we consider the following perturbed equation:
\begin{equation*}
\left\{
\begin{aligned}
& (i\partial _t + \Delta ) v = \pm \mathcal{N}_{m}(v+z) , \\
& v(0,x) =0 .
\end{aligned}
\right.
\end{equation*}
In the previous results of B\'{e}nyi et. al. \cite{BOP2} and the authors \cite{HO15}, the lower bound of $s$ comes from a nonlinearity part which only consists of the linear evolution of the probabilistic initial data.
However, the lower bound in Theorems \ref{thm:LWP} and \ref{thm:GWP} appears in different nonlinear parts when $m \ge 3$.
More precisely, $\frac{d-1}{d} s_c$ and $s_c-\frac{d}{2(d+1)}$ are need to estimate $\mathcal{N}_m(z \cdots zv)$ and $\mathcal{N}_m(v \cdots vz)$ respectively.
Hence, $v$, which has more regularity than $z$, behaves like a bad part for $d\ge 2$ and $m \ge 4$.

We now give a brief outline of this article.
In Section \ref{S:prop_lem}, we collect lemmas which are used in the proof of our main results.
In Section \ref{S:funct_sp}, we define the function spaces and show these properties.
In Section \ref{S:nonlinear}, we show that the nonlinear estimates, which play a crucial role in the proof of our main results.
In Section \ref{S:proof_WP}, we give a proof of almost sure well-posedness results, Theorems \ref{thm:LWP} and \ref{thm:GWP}.

\section{The probabilistic lemmas} \label{S:prop_lem}

Firstly, we recall the probabilistic estimate.
The randomization keeps differentiability of the function.

\begin{lem}[\cite{BOP1}] \label{lem:randiff}
Given $\phi \in H^s (\R ^d)$, let $\phi ^{\omega}$ be its randomization defined by \eqref{eq:rand}.
Then, there exist $C,c >0$ such that
\[
P( \norm{\phi ^{\omega}}_{H^s} > \lambda ) < C \exp \left( - c \frac{\lambda ^2}{\norm{\phi}_{H^s}^2} \right)
\]
for all $\lambda >0$.
\end{lem}

Let $S(t) := e^{it \Delta}$ be the linear propagator of the Schr\"{o}dinger group,
Namely, $v(t) = S(t) \phi$ solves
\[
(i\partial _t + \Delta ) v= 0, \quad v(0) = \phi.
\]
We say that a pair $(q,r)$ is admissible if $2 \le q,r \le \infty$, $(q,r,d) \neq (2, \infty ,2)$, and
\[
\frac{2}{q} + \frac{d}{r} = \frac{d}{2}.
\]
The following Strichartz estimates hold.

\begin{prop} \label{prop:Str}
Let $(q,r)$ be admissible.
Then, we have
\[
\norm{S(t) \phi}_{L_t^q L_x^r} \lesssim \norm{\phi}_{L^2_x}.
\]
\end{prop}

By the randomization, improved Strichartz type estimates hold.

\begin{lem}[\cite{BOP1}] \label{lem:stcStr}
Given $\phi \in L^2(\R ^d)$, let $\phi ^{\omega}$ be its randomization defined by \eqref{eq:rand}.
Let $(q,r)$ be admissible with $q,r<\infty$ and $r \le \bar{r} < \infty$.
Then, there exist $C, c>0$ such that
\[
P(\norm{S(t) \phi ^{\omega}}_{L_t^q L_x^{\bar{r}}} > \lambda ) \le C \exp \left( -c \frac{\lambda ^2}{\norm{\phi}_{L^2_x}^2} \right)
\]
for all $\lambda >0$.
\end{lem}

\section{Function spaces and their properties} \label{S:funct_sp}

\subsection{Definition of $U^{p},$ $V^{p}$ spaces}

In this section, we define the $U^{p}$- and $V^{p}$-type function spaces.
We refer the reader to \S 2 in \cite{HHK} for proofs of the basic properties.

Let $\mathcal{Z}$ be the set of finite partitions $-\infty <t_{0}<t_{1}<\dots <t_{K} \le \infty $ of the real line and we put $v(t_K):=0$ for all functions $v:\R\rightarrow L^2$ if $t_K=\infty$.

\begin{dfn}
Let $1\leq p<\infty .$ For $\{t_{k}\}_{k=0}^{K}\in \mathcal{Z}$ and $\{\phi _{k}\}_{k=0}^{K-1}\subset L^2 (\R ^d)$ with 
$\sum_{k=0}^{K-1} \norm{\phi _{k}}_{L^{2}}^{p}=1,$ we call the function $a: \R \rightarrow L^2 (\R ^d)$ given by 
\begin{equation*}
a=\sum_{k=1}^{K}\chi _{[t_{k-1},t_{k})}\phi _{k-1}
\end{equation*}
a $U^{p}$-atom. Furthermore, we define the atomic space 
\begin{equation*}
U^{p}:=\left\{ u: \R \rightarrow L^2 (\R ^d) : u=\sum_{j=1}^{\infty }\lambda _{j}a_{j}\ \text{for $U^{p}$-atoms $a_{j}$, $\{ \lambda _{j}\}\subset \C$ such that $\sum_{j=1}^{\infty }|\lambda _{j}|<\infty $}\right\}
\end{equation*}
with the norm 
\begin{equation*}
\norm{u} _{U^{p}}:=\inf \left\{ \sum_{j=1}^{\infty }|\lambda _{j}|:u=\sum_{j=1}^{\infty }\lambda _{j}a_{j}\ \text{for $U^{p}$-atoms $a_{j} $, $\{\lambda _{j}\}\subset \C$ such that $\sum_{j=1}^{\infty }|\lambda _{j}|<\infty $}\right\} .
\end{equation*}
\end{dfn}

\begin{dfn}
(i) Let $1\le p<\infty .$ We define $V^{p}$ as the space of all functions $v: \R \rightarrow L^2 (\R ^d)$ such that the limits $\lim _{t \to \pm \infty}v(t)$ exist in $L^2 (\R ^d)$ and 
the norm 
\begin{equation}
\norm{v}_{V^{p}}:=\sup_{\{t_{k}\}_{k=0}^{K}\in \mathcal{Z}}\left( \sum_{k=1}^{K}\Vert v(t_{k})-v(t_{k-1})\Vert _{L^{2}}^{p}\right) ^{1/p}
\label{def_Vp}
\end{equation}
is finite. 

(ii) Let $V_{-,rc}^{p}$ be the closed subspace of all $v\in V^{p}$ such that $v$ is right continuous and $\lim_{t\rightarrow -\infty }v(t)=0$, endowed with the norm (\ref{def_Vp}).
\end{dfn}

For $1\leq p<q<\infty $, $U^p \hookrightarrow V_{-,rc}^{p}\hookrightarrow U^{q} \hookrightarrow L_{t}^{\infty }(\R ;L_x^2 (\R ^d))$ is valid.

Using the $U^2$ and $V^2$ spaces instead of $H_t^{1/2+\eps} (\R ;L^2_x(\R ^d)) (\hookrightarrow C(\R ;L_x^2 (\R ^d)))$, we define the Fourier restriction norm spaces.

\begin{dfn}
(i) Let $1\leq p<\infty $.
We define the function spaces $U_{\Delta}^{p}:=S (t) U^{p}$ (resp., $V_{\Delta}^{p}:=S(t)V^{p}$) as
the spaces of all functions $u:\R \rightarrow L^2 (\R ^d)$ such that $t \rightarrow S \left( -t \right) u(t)$ is in $U^{p}$ (resp., $V^{p}$),
with the norms 
\begin{equation*}
\norm{u}_{U_{\Delta}^{p}}:=\norm{S(-\cdot )u} _{U^{p}}, \quad
\norm{v}_{V_{\Delta}^{p}}=\norm{S(-\cdot )v} _{V^{p}}.
\end{equation*}

(ii) The closed subspace $V_{-,rc,\Delta}^{p}$ is defined similarly.
\end{dfn}

The Strichartz estimates imply the following.

\begin{lem} \label{lem:Stv}
Let $d \ge 1$ and let $(q,r)$ be admissible with $q>2$.
Then, we have
\[
\norm{u}_{L_t^q L_x^r} \lesssim \norm{u}_{V^2_{\Delta}}.
\]
\end{lem}

We use the convention that capital letters denote dyadic numbers, e.g., $N=2^{n}$ for $n\in \mathbb{N}_{0}:=\mathbb{N}\cup \{ 0\}$.
We fix a nonnegative even function $\varphi \in C_{0}^{\infty }((-2,2))$ with $\varphi (r)=1$ for $|r|\leq 1$ and $\varphi \left( r\right) \leq 1$ for $1\le |r| \le 2$.
Set $\varphi _{N}(r):=\varphi (r/M) - \varphi (2r /N) $ for $N\geq 2$ and $\varphi _{1}(r):=\varphi (r)$.
For $N\in 2^{\mathbb{N}_{0}},$ $P_{N}$ denotes the Fourier multiplier with the symbol $\varphi _{N}(|\xi |),$ i.e. $(P_{N}f)(x):=\mathcal{F} ^{-1}[\varphi _{N}(|\xi |)\hat{f}(\xi)](x)$. 
Define $P_{>N}:=\sum_{M>N}P_{M}$ and $P_{\leq N}:=\Id-P_{>N}$.
Moreover, for $M\in 2^{\mathbb{N}_0},$ we define $\widetilde{Q _{M} f}(\tau ,\xi ):= \varphi _{M}(\tau +|\xi |^{2}) \tilde{f}(\tau ,\xi )$.
We also use $Q _{>M} :=\sum_{N>M} Q_{N}^{n}$ and $Q_{\le M} :=\Id -Q _{>M}^{n}.$

We state the boundedness of the operators $Q_{>M}$ and $Q_{\le M}$.

\begin{lem} \label{lem:Q}
Let $d\geq 1$, $2\le p \le \infty$, and $M\in 2^{\mathbb{N}_0}$.
Then the following estimates
\begin{equation*}
\norm{Q _{>M} f}_{L_t^p L_x^2} \lesssim M^{-1/p} \norm{f}_{V_{\Delta}^{2}}, \quad
\norm{Q _{\le M} f}_{V^2_{\Delta}} + \norm{Q _{>M} f}_{V^2_{\Delta}} \lesssim \norm{f}_{V_{\Delta}^{2}},
\end{equation*}
hold for any $f\in V_{\Delta}^{2},$ where the implicit constants are dependent only on $d.$
\end{lem}

The bilinear refinement of the Strichartz estimate holds (\cite{Bou98}, \cite{CDKS}, \cite{CKSTT}).

\begin{lem} \label{lem:biSt}
Let $d \ge 1$ and let $N_1, \, N_2, \, N_3 \in 2^{\N _0}$.
Assume $N_{\min}=\min (N_1,N_2,N_3) \ll N_{\max}=\max (N_1,N_2,N_3)$.
Then, we have
\[
\norm{P_{N_3}\left(S(t) P_{N_1} \phi _1 S(t) P_{N_2} \phi _2\right)}_{L^2_{t,x}} \lesssim N_{\min}^{\frac{d}{2}-1} \left( \frac{N_{\min}}{N_{\max}} \right) ^{\frac{1}{2}} \norm{\phi _1}_{L^2} \norm{\phi _2}_{L^2}
\]
for any $\phi _1, \, \phi _2 \in L^2(\R ^d)$.
\end{lem}

Combining the interpolation (see Proposition 2.20 in \cite{HHK}) with it, we obtain the bilinear refinement of the Strichartz estimate in the $V^2$ space settings.

\begin{cor} \label{cor:biStv}
Let $d \ge 1$.
For any $(v,w) \in V_{-,rc, \Delta}^{2} \times V_{-,rc,\Delta}^{2}$, $N_1, \, N_2, \, N_3 \in 2^{\N _0}$ with $N_{\min} \ll N_{\max}$, and sufficiently small $\delta >0$, we have the estimate
\[
\norm{P_{N_3}\left(P_{N_1}v P_{N_2}w\right)}_{L^2_{t,x}} \lesssim N_{\min}^{\frac{d}{2}-1} \left( \frac{N_{\min}}{N_{\max}} \right) ^{\frac{1}{2}-\delta} \norm{v}_{V^2_{\Delta}} \norm{w}_{V^2_{\Delta}}
\]
where the implicit constants depending only on $d$.
\end{cor}

\begin{rmk}
By the Strichartz estimate, the same estimate holds in the case $N_{\min} \sim N_{\max}$ except for $d=1$.
Hence, we neglect the condition $N_{\min} \ll N_{\max}$ if $d \ge 2$.
\end{rmk}

\begin{dfn} \label{def:fs}
For $s\in \mathbb{R}$, we define $Y^s$ and $Z^s$ as the closure of $C(\R ;\mathcal{S}(\R ^d))\cap V_{-,\Delta}^2$ and $C(\R ;\mathcal{S}(\R ^d)) \cap U_{\Delta}^2 $ with respect to the norm
\begin{equation*}
\norm{f}_{Y^s} :=\bigg( \sum_{N \in 2^{\N _0}}N^{2s} \Norm{P_{N}f}_{V_{\Delta}^2}^{2}\bigg) ^{1/2} , \quad
\norm{f}_{Z^s} :=\bigg( \sum_{N \in 2^{\N _0}}N^{2s} \Norm{P_{N}f}_{U_{\Delta}^2}^{2}\bigg) ^{1/2} ,
\end{equation*}
respectively.
\end{dfn}

We also use the time restricted space.

\begin{dfn} \label{def:time_tr}
Let $E$ be a Banach space of continuous functions $f: \R \rightarrow H$ for some Hilbert space $H$.
We define the corresponding restriction space to the interval $[0,T) \subset \R$ as
\begin{equation*}
E_{T}:=\{f\in C([0,T);H):\exists g^{\ast }\in E,\ g^{\ast }(t)=f(t),t\in [0,T)\}
\end{equation*}
endowed with the norm $\norm{f}_{E_{T}} :=\inf \{ \norm{g^{\ast }}_{E}:g^{\ast }(t)=f(t),t\in [0,T)\}$. 
\end{dfn}
The space $E_{T}$ is a Banach space.
For any $T \in ( 0,\infty ]$, we have the embeddings
\[
Z_T^s \hookrightarrow Y_T^s \hookrightarrow \langle \nabla \rangle ^{-s} V^2_{\Delta ,T} \cap C([0,T);H^s).
\]

Let $f\in L_{loc}^{1}\left( [0,\infty );L_x^2 (\R ^d) \right)$.
We define the integral operator
\begin{equation*}
\Gamma [f] (t) :=\int_{0}^{t} S ( t-t') f (t') dt' ,
\end{equation*}
for $t\ge 0$ and $\Gamma [f] (t) =0$ otherwise.
For the integral operator, we have the following.

\begin{prop} \label{prop:Duh}
Let $d\geq 1$, $s \in \R$, and $T\in ( 0,\infty ]$.
Then the estimate 
\begin{equation*}
\norm{\Gamma [f]}_{Z_{T}^{s}}
\le \left\{ \sum_{N \in 2^{\N _0}} N^{2s} \left( \sup_{\norm{g}_{V_{\Delta}^{2}=1}} \left| \int_{0}^{T} \lr{f(t),P_N g(t)}_{L_{x}^{2}} dt \right| \right) ^{2}\right\} ^{1/2}
\end{equation*}
holds for any $f\in L_t^1((0,T);H^{s}(\R ^d)),$ where the implicit constant is dependent only on $d,s.$
\end{prop}

This estimate follows from Proposition 2.10, Remark 2.11 in \cite{HHK}.

\section{Probabilistic nonlinear estimates} \label{S:nonlinear}

First of all, we recall the notations which are introduced in \S 1.
Put $\mathcal{N}_{m} (u) = \partial (\overline{u}^{m})$.
Let $z(t) = z^{\omega} (t) := S(t) \phi ^{\omega}$ and $v(t) = u(t) - z(t)$ be the linear and nonlinear parts of $u$ respectively.
We consider the following perturbed equation:
\begin{equation*}
\left\{
\begin{aligned}
& (i\partial _t + \Delta ) v = \pm \mathcal{N}_{m}(v+z) , \\
& v(0,x) =0 .
\end{aligned}
\right.
\end{equation*}

To state probabilistic nonlinear estimates, we define the following sets:
\begin{align*}
\mathfrak{S}^{2}_{\delta} := & \left\{ (4,4) ,  (4,2d) \right\} ,
\\
\mathfrak{S}^{m}_{\delta} := & \left\{ (4,4) , \big( 4(m-1), \tfrac{2(m-1)^2d}{(m-1)d+m-3} \big) , \big( \tfrac{2(m-2)(2+\delta )}{\delta} , \tfrac{2(m-1)(m-2)d}{(m-1)d-2} \big) , \big( 4, \tfrac{2d}{d-1} \big) \right\} \\
& \cup \bigcup _{l=1}^{m-1} \left\{ \big( \tfrac{2(m-l)(2+\delta )}{\delta}, \tfrac{2(m-l)(2+\delta )}{\delta} \big) , \big( \tfrac{4(m-l)(4+\delta )}{\delta}, \tfrac{2(m-1)(m-l)(4+\delta )d}{(m-1)(8+\delta )-2(4+\delta )(l-1)} \big) , \big( \tfrac{2(m-l)(2+\delta )}{\delta}, \tfrac{m-l}{\delta} \big) , \right. \\
& \hspace*{40pt} \left. \big( \tfrac{2(m-l)(2+\delta )}{\delta}, \tfrac{d(m-l)(d+\delta )}{\delta} \big) , \big( \tfrac{2(m-l)(2+\delta )}{\delta}, (m-1)d \big)  \right\} .
\end{align*}
The set $\mathfrak{S}^{2}_{\delta}$ does not depend on $\delta$.
But for convenience, we use this notation.
For an interval $I \subset \R$ and $\delta >0$,
\begin{equation} \label{norm:S_0}
\norm{u}_{S_{\delta}^{m} (I)} := \max \{ \norm{u}_{L_t^q L_x^r (I \times \R ^d)} : (q,r) \in \mathfrak{S}^{m}_{\delta} \} .
\end{equation}
The followings are the main results in this section.

\begin{lem} \label{lem:nonest}
Assume  $d \ge 1$, $m \ge 2$, and $d+m \ge 5$.
Let $\max \left( \frac{d-1}{d} s_c , \frac{s_c}{2}, s_c - \frac{d}{2(d+1)} \right)<s< s_c$ and $\delta >0$ be sufficiently small depending only on $d$, $m$ and $s$.
Given $\phi \in H^s(\R ^d)$, let $\phi ^{\omega}$ be its randomization defined by \eqref{eq:rand}.
For $R>0$, we put
\[
E_R^m := \{ \omega \in \Omega : \norm{\phi ^{\omega}}_{H^s} + \norm{\lr{\nabla}^s S(t) \phi ^{\omega}}_{S^m_{\delta} (\R )} \le R \} .
\]
Then, we have
\begin{align}
\norm{\Gamma [\mathcal{N}_{m}(v+z)]}_{Z^{s_c}_T} & \le C_1 \left( \norm{v}_{Y^{s_c}_T}^m + R ^m \right) , \label{eq:nonest1} \\
\norm{\Gamma [\mathcal{N}_{m}(v_1+z)] - \Gamma [\mathcal{N}(v_2+z)]}_{Z^{s_c}_T} & \le C_2 \left( \norm{v_1}_{Y^{s_c}_T}^{m-1} + \norm{v_2}_{Y^{s_c}_T}^{m-1} + R^{m-1} \right) \norm{v_1-v_2}_{Y^{s_c}_T} \label{eq:nonest2}
\end{align}
for any $T \in (0, \infty )$, $v, v_1, v_2 \in Y^{s_c}_T$, and $\omega \in E_R^m$.
Here, the constants $C_1$ and $C_2$ are depending only on $d$ and $m$.
\end{lem}

\begin{rmk}
In the quadratic case, the condition $s> \frac{d-1}{d} s_c$ comes from the estimate for $zz$.

In the cubic case, the condition $s>\frac{d-1}{d} s_c$ needs to treat the $zzv$ case (see \S \ref{S:Nonestm3} below), while the other cases are less restricted.
On the other hand, the lower bound of the regularity in \cite{BOP2} appears in the $zzz$ case.
\end{rmk}

\begin{rmk} \label{rmk:outside1}
Note that the pairs 
\[
\big( \tfrac{2(m-l)(2+\delta )}{\delta}, \tfrac{2d(m-l)(2+\delta )}{d(m-l)(2+\delta )-2\delta} \big) ,\
\big( 4(m-1), \tfrac{2d(m-1)}{(m-1)d-1} \big) , \big( \tfrac{4(m-l)(4+\delta )}{\delta} , \
\tfrac{2d(m-l)(4+\delta )}{d(m-l)(4+\delta )-\delta} \big) , \
\big( 4, \tfrac{2d}{d-1} \big)
\]
are admissible.
Accordingly, Lemmas \ref{lem:randiff} and \ref{lem:stcStr} imply that $E_R$ in Lemma \ref{lem:nonest} satisfies the bound
\[
P(\Omega \backslash E_R^m) \le C \exp \left( -c \frac{R^2}{\norm{\phi}_{H^s}^2} \right).
\]
\end{rmk}

To show the local in time nonlinear estimates, we define the norm
\begin{equation} \label{norm:S_0'}
\norm{u}_{S_{\delta}^{m,L} (I)} := \max \Big\{ \norm{u}_{L_t^q L_x^r (I \times \R ^d)} : \big( \tfrac{q}{1-\delta q} ,r \big) \in \mathfrak{S}_{\delta}^m \backslash \{ (4,4), \big( 4, \tfrac{2d}{d-1} \big) \} \text{ or } (q,r)=(4,4), \big( 4, \tfrac{2d}{d-1} \big) \Big\}
\end{equation}
Since H\"{o}lder's inequality yields
\begin{equation} \label{boundT}
\norm{f}_{L^q_T X} \lesssim T^{\delta} \norm{f}_{L^{\frac{q}{1-\delta q}}_T X}
\end{equation}
for any Banach space $X$, we obtain the following (see the proof of Lemma \ref{lem:nonest} and Remark \ref{rmkm2} below).

\begin{lem} \label{lem:nonestl}
Assume $d \ge 1$, $m \ge 2$, and $d+m \ge 5$.
Let $\max \left( \frac{d-1}{d} s_c , \frac{s_c}{2}, s_c - \frac{d}{2(d+1)} \right)<s< s_c$ and $\delta >0$ be sufficiently small depending only on $d$, $m$ and $s$.
Given $\phi \in H^s(\R ^d)$, let $\phi ^{\omega}$ be its randomization defined by \eqref{eq:rand}.
For $R>0$, we put
\[
E_R^{m,L} := \{ \omega \in \Omega : \norm{\phi ^{\omega}}_{H^s} + \norm{S(t) \phi ^{\omega}}_{S_{\delta}^{m,L} (\R )} \le R \} .
\]
Then, we have
\begin{align}
\norm{\Gamma [\mathcal{N}_m(v+z)]}_{Z^{s_c}_T} & \le C_1' \left( \norm{v}_{Y^{s_c}_T}^m + T ^{\delta} R ^m \right) , \label{eq:nonestl1} \\
\norm{\Gamma [\mathcal{N}_m(v_1+z)] - \Gamma [\mathcal{N}_m(v_2+z)]}_{Z^{s_c}_T} & \le C_2' \left( \norm{v_1}_{Y^{s_c}_T}^{m-1} + \norm{v_2}_{Y^{s_c}_T}^{m-1} + T^{\delta} R^{m-1} \right) \norm{v_1-v_2}_{Y^{s_c}_T} \label{eq:nonestl2}
\end{align}
for $0<T \le 1$, all $v, v_1, v_2 \in Y^{s_c}_T$, and $\omega \in E_R^{m,L}$.
\end{lem}

Lemmas \ref{lem:randiff} and \ref{lem:stcStr} imply the bound (see Remark \ref{rmk:outside1})
\begin{equation} \label{outside2}
P(\Omega \backslash E_R^{m,L}) \le C \exp \left( -c \frac{R^2}{\norm{\phi}_{H^s}^2} \right) .
\end{equation}

\noindent
\textit{Proof of Lemma \ref{lem:nonest}}

We only prove \eqref{eq:nonest1} because \eqref{eq:nonest2} follows from a similar manner.
Thanks to Proposition \ref{prop:Duh}, it suffices to show
\begin{equation} \label{qu_est1}
\left\{ \sum _{N_0 \in 2^{\N _0}} N_0^{2s_c+2} \left| \int _{\R ^{1+d}} P_{N_0} v_0 \prod _{j=1}^m w_j dx dt \right| ^2 \right\} ^{\frac{1}{2}}
\lesssim \norm{v}_{Y^{s_c}}^m + R^m,
\end{equation}
where $w_j= v$ or $z$ ($j=1, \dots , m$) and $v_0 \in V^2_{\Delta}$ with $\norm{v_0}_{V^2_{\Delta}} =1$.

We use the dyadic decomposition as follows.
\[
\int _{\R ^{1+d}} P_{N_0} v_0 \prod _{j=1}^m w_j dx dt
= \sum _{N_1, \dots , N_m \in 2^{\N _0}} \int _{\R ^{1+d}} P_{N_0} v_0 \prod _{j=1}^m P_{N_j} w_j dx dt
\]
Here, we divide the integration on the right hand side into $2^{m+1}$ parts of the form
\[
\int _{\R ^{1+d}} Q_0 P_{N_0} v_0 \prod _{j=1}^m Q_j P_{N_j} w_j dx dt
\]
with $Q_j  \in \{ Q _{\le M}, Q_{>M} \}$ ($j=0,1, \dots ,m$).
This decomposition is only meaningful if $w_j = v$ because $Q_{\le M} z = Q_{\le M} e^{it\Delta} \phi ^{\omega} = e^{it\Delta} \phi ^{\omega} =z$ and $Q_{>M} z = Q_{>M} e^{it \Delta} \phi ^{\omega} =0$.

Here, we note that at least one of the modulations is bounded below.
More precisely, for $(\tau _j ,\xi _j) \in \R ^{1+d}$ ($j=0,1, \dots ,m$) with $\sum _{j=0}^m \tau _j =0$ and $\sum _{j=0}^m \xi _j =0$, by the triangle inequality, we have
\begin{equation} \label{mod}
\max _{0 \le j \le m} |\tau _j + |\xi _j|^2| \ge \frac{1}{m+1} \max _{0 \le j \le m} |\xi _j|^2.
\end{equation}
Thus, let us assume that one of $Q_j$ is $Q_{> \max _{0 \le j \le m} N_j^2}$ otherwise the integration becomes zero.
Putting $\mathfrak{R}_j := Q_j P_{N_j}$, we mainly focus on the estimate of
\[
I := \sum _{N_0, N_1 \dots , N_m \in 2^{\N _0}} N_0^{s_c+1} \left| \iint _{\R ^{1+d}} \mathfrak{R}_0 v_0 \prod _{j=1}^m \mathfrak{R}_j w_j dx dt \right| ,
\]
which is bigger than the left hand side of \eqref{qu_est1} because of $l^1 \hookrightarrow l^2$.
For a set $\mathfrak{N} \subset (2^{\N _0})^{m+1}$ (for example, $\mathfrak{N}$ is defined by $\{ N_2, \dots , N_m \le N_0 \sim N_1 \}$), we use the notation $I_{\mathfrak{N}}$ as
\[
I_{\mathfrak{N}} := \sum _{\substack{N_0, N_1, \dots , N_m \in 2^{\N _0} \\ (N_0, N_1, \dots , N_m) \in \mathfrak{N}}} N_0^{s_c+1} \left| \iint _{\R ^{1+d}} \mathfrak{R}_0 v_0 \prod _{j=1}^m \mathfrak{R}_j w_j dx dt \right| .
\]

We separately treat the cases $m=2$ and $m \ge 3$.

\subsection{The case $m=2$}

In this subsection, we consider the case $m = 2$, where we have $s_c = \frac{d-2}{2}$.

\begin{proof}[Proof of \eqref{eq:nonest1} with $m=2$]

\mbox{}

\noindent
\textbf{Case 1:} $vv$ case.

Although this case is essentially treated in \cite{H14}, we give a proof for completeness.
Put $v_1=v_2=v$, $N_{\min}=\min (N_0,N_1,N_2)$, $N_{\med}=\med (N_0,N_1,N_2)$, and $N_{\max}=\max (N_0,N_1,N_2)$ for convenience.
There exists a permutation $\{ i,j,k \}$ of $\{ 0,1,2 \}$ such that $Q_{i} = Q_{>N_{\max}^2}$.
By Lemma \ref{lem:Q}, Corollary \ref{cor:biStv}, and $\norm{v_0}_{V^2_{\Delta}} =1$, we have
\[
\left| \int _{\R ^{1+d}} \mathfrak{R}_0 v_0 \mathfrak{R}_1 v \mathfrak{R}_2 v dx dt \right|
\lesssim \norm{Q_{N_{\max}^2} P_{N_i} v_i}_{L^2_{t,x}} \norm{P_{N_i} (\mathfrak{R}_j v_j \mathfrak{R}_k v_k)}_{L^2_{t,x}}
\lesssim N_{\max}^{-\frac{3}{2}+\delta} N_{\min}^{\frac{d-1}{2}-\delta} \prod _{l=0}^2 \norm{P_{N_l} v_l}_{V^2_{\Delta}}.
\]
Thanks to $N_{\min} \lesssim N_{\min} \sim N_{\max}$, we obtain
\begin{align*}
& \sum _{N_1, N_2 \in 2^{\N _0}} \left\{ \sum _{\substack{N_0 \in 2^{\N _0} \\ N_{\min} \lesssim N_{\med} \sim N_{\max}}} N_0^{d} \left| \int _{\R ^{1+d}} \mathfrak{R}_0 v_0 \mathfrak{R}_1 v \mathfrak{R}_2 v dx dt \right| ^2 \right\} ^{\frac{1}{2}} \\
& \lesssim \sum _{N_1, N_2 \in 2^{\N _0}} \left\{ \sum _{\substack{N_0 \in 2^{\N _0} \\ N_{\min} \lesssim N_{\med} \sim N_{\max}}} N_{\max}^{d-2} N_{\min}^{d-2} \prod _{l=0}^2 \norm{P_{N_l} v_l}_{V^2_{\Delta}}^2 \right\} ^{\frac{1}{2}} \\
& \lesssim \norm{v}_{Y^{\frac{d-2}{2}}}^2.
\end{align*}

\noindent
\textbf{Case 2:} $zz$ case.

Without loss of generality, we may assume $N_1 \le N_2$.
Moreover, $Q_0 =Q_{> N_0^2}$ holds in this case.

\textbf{Subcase 2-1:} $N_1 \ll N_2 \sim N_0$, $N_1\lesssim N_2^{\frac{1}{d-1}}$.

By H\"{o}lder's inequality, and Lemmas \ref{lem:Q}, \ref{lem:biSt}, we get
\begin{align*}
I_{\substack{N_1 \ll N_2 \sim N_0, \\N_1\lesssim N_2^{\frac{1}{d-1}}}}
& \lesssim \sum _{\substack{N_1 \ll N_2 \sim N_0 \\ N_1\lesssim N_2^{\frac{1}{d-1}}}} 
N_0^{\frac{d}{2}}\norm{Q_{>N_0^2} P_{N_0} v_0}_{L^{2}_{t,x}} \norm{P_{N_1} z P_{N_2} z }_{L^2_{t,x}} \\
& \lesssim \sum _{\substack{N_1 \ll N_2 \sim N_0 \\ N_1\lesssim N_2^{\frac{1}{d-1}}}} 
N_0^{\frac{d}{2}} N_0^{-1}N_1^{\frac{d}{2}-1}\left(\frac{N_1}{N_2}\right)^{\frac{1}{2}}
\norm{P_{N_0} v_0}_{V^{2}_{\Delta}}\norm{P_{N_1}\phi^{\omega}}_{L^2_x}\norm{P_{N_2}\phi^{\omega}}_{L^2_x}\\
& \lesssim \sum _{\substack{N_1 \ll N_2 \sim N_0 \\ N_1\lesssim N_2^{\frac{1}{d-1}}}} 
N_1^{\frac{d-1}{2}}N_2^{\frac{d-3}{2}}N_1^{-s}N_2^{-s}R^2\norm{P_{N_0} v_0}_{V^{2}_{\Delta}}\\
& \lesssim \sum _{\substack{N_1 \ll N_2 \sim N_0 \\ N_1\lesssim N_2^{\frac{1}{d-1}}}} 
N_1^{-s+\frac{(d-1)(d-2)}{2d}}N_2^{-s+\frac{(d-1)(d-2)}{2d}}R^2\norm{P_{N_0} v_0}_{V^{2}_{\Delta}}\\
&\lesssim R^2
\end{align*}
for $\omega \in E_R^2$.
Here, we have used $\frac{(d-1)(d-2)}{2d} <s<\frac{d-2}{2}$, $\norm{v_0}_{V^2_{\Delta}} =1$ and $\delta >0$ is sufficiently small in the last inequality.

\textbf{Subcase 2-2:} $N_1 \ll N_2 \sim N_0$, $N_1\gtrsim N_2^{\frac{1}{d-1}}$.

By H\"{o}lder's inequality and Lemma \ref{lem:Q}, we get
\begin{align*}
I_{\substack{N_1 \ll N_2 \sim N_0, \\N_1\gtrsim N_2^{\frac{1}{d-1}}}}
& \lesssim \sum _{\substack{N_1 \ll N_2 \sim N_0 \\ N_1\gtrsim N_2^{\frac{1}{d-1}}}} 
\norm{Q_{>N_0^2} P_{N_0} v_0}_{L^{2}_{t,x}} N_0^{\frac{d}{2}}\norm{P_{N_1} z }_{L^4_{t,x}}\norm{P_{N_2} z }_{L^4_{t,x}} \\
& \lesssim \sum _{\substack{N_1 \ll N_2 \sim N_0 \\ N_1\gtrsim N_2^{\frac{1}{d-1}}}} 
N_0^{\frac{d}{2}} N_0^{-1}N_1^{-s}N_2^{-s}R^2 \norm{P_{N_0} v_0}_{V^{2}_{\Delta}}\\
& \lesssim \sum _{\substack{N_1 \ll N_2 \sim N_0 \\ N_1\gtrsim N_2^{\frac{1}{d-1}}}} 
N_1^{-s+\frac{(d-1)(d-2)}{2d}}N_2^{-s+\frac{(d-1)(d-2)}{2d}}R^2 \norm{P_{N_0} v_0}_{V^{2}_{\Delta}}\\
&\lesssim R^2
\end{align*}
for $\omega \in E_R^2$.
Here, we have used $\frac{(d-1)(d-2)}{2d} <s<\frac{d-2}{2}$ and $\norm{v_0}_{V^2_{\Delta}} =1$ in the last inequality.

\textbf{Subcase 2-3:} $N_0 \lesssim N_1 \sim N_2$.

By H\"{o}lder's inequality and Lemma \ref{lem:Q}, we get
\begin{align*}
I_{N_0 \lesssim N_1 \sim N_2}
& \lesssim \sum _{N_0 \lesssim N_1 \sim N_2} 
N_0^{\frac{d}{2}}\norm{Q_{>N_0^2} P_{N_0} v_0}_{L^{2}_{t,x}} \norm{P_{N_1} z }_{L^4_{t,x}}\norm{P_{N_2} z }_{L^4_{t,x}} \\
& \lesssim \sum _{N_0 \lesssim N_1 \sim N_2} 
N_0^{\frac{d}{2}} N_0^{-1}N_1^{-s}N_2^{-s}R^2 \norm{P_{N_0} v_0}_{V^{2}_{\Delta}}\\
& \lesssim \sum _{N_0} 
N_{0}^{-2s+\frac{d-2}{2}}R^2 \norm{P_{N_0} v_0}_{V^{2}_{\Delta}}\\
&\lesssim R^2
\end{align*}
for $\omega \in E_R^2$.
Here, we have used $\frac{d-2}{4}\le \frac{(d-1)(d-2)}{2d} <s<\frac{d-2}{2}$ and $\norm{v_0}_{V^2_{\Delta}} =1$ in the last inequality.

\noindent
\textbf{Case 3:} $zv$ case.

We consider only $N_2\lesssim N_0$ since the case $N_0\lesssim N_2$ is simpler.
(In fact, if $N_0\lesssim N_2$, then $N_0^{\frac{d}{2}}N_2^{-s}\lesssim N_2^{\frac{d}{2}}N_0^{-s}$ for $s\ge 0$.)
In this case, $Q_2 =Q_{> N_0^2}$ or $Q_0 =Q_{> N_0^2}$ holds in this case. 
We deal with only $Q_2 =Q_{> N_0^2}$ because the case $Q_0 =Q_{> N_0^2}$ follows from the same manner.

\textbf{Subcase 3-1:} $N_2\ll N_1\sim N_0$, $N_2\lesssim N_0^{\frac{1}{d-1}}$.

By H\"{o}lder's inequality, Lemma \ref{lem:Q}, and Corollary \ref{cor:biStv}, we have
\begin{align*}
I_{N_2\ll N_1\sim N_0, N_2\lesssim N_0^{\frac{1}{d-1}}}
& \lesssim \sum _{\substack{N_2\ll N_1\sim N_0 \\ N_2\lesssim N_0^{\frac{1}{d-1}}}} 
N_0^{\frac{d}{2}} \norm{P_{N_2}\left(P_{N_1} z \mathfrak{R}_0 v_0 \right)}_{L^2_{t,x}}  \norm{Q_{>N_0^2} P_{N_2} v}_{L^2_{t,x}} \\
& \lesssim \sum _{\substack{N_2\ll N_1\sim N_0 \\ N_2\lesssim N_0^{\frac{1}{d-1}}}} 
N_0^{\frac{d}{2}}N_0^{-1} N_2^{\frac{d}{2}-1}\left(\frac{N_2}{N_1}\right)^{\frac{1}{2}-\delta}
\norm{P_{N_1} \phi^{\omega}}_{L^{2}_x}  \norm{P_{N_2} v}_{V^2_{\Delta}} \norm{P_{N_0}v_0}_{V^2_{\Delta}} \\
& \lesssim \sum _{\substack{N_2\ll N_1\sim N_0 \\ N_2\lesssim N_0^{\frac{1}{d-1}}}} 
N_1^{\frac{d-3}{2}+\delta} N_2^{\frac{1}{2}-\delta}N_1^{-s}
R(N_2^{\frac{d-2}{2}}  \norm{v}_{V^2_{\Delta}}) \norm{P_{N_0}v_0}_{V^2_{\Delta}} \\
& \lesssim \sum _{N_1\sim N_0} 
N_1^{-s+\frac{(d-2)^2}{2(d-1)}+\frac{d-2}{d-1}\delta} 
R \norm{v}_{Y^{\frac{d-2}{2}}} \norm{P_{N_0}v_0}_{V^2_{\Delta}} \\
& \lesssim 
\norm{ v}_{Y^{\frac{d-2}{2}}}R
\end{align*}
for $\omega \in E_R^2$.
Here, we have used $\frac{(d-2)^2}{2(d-1)}\le \frac{(d-1)(d-2)}{2d} <s<\frac{d-2}{2}$, $\norm{v_0}_{V^2_{\Delta}} =1$ and $\delta >0$ is sufficiently small in the last inequality.

\textbf{Subcase 3-2:} $N_2\ll N_1\sim N_0$, $N_2\gtrsim N_0^{\frac{1}{d-1}}$.

By H\"{o}lder's inequality, Lemmas \ref{lem:Stv} and \ref{lem:Q}, we have
\begin{align*}
I_{N_2\ll N_1\sim N_0, N_2\gtrsim N_0^{\frac{1}{d-1}}}
& \lesssim \sum _{\substack{N_2\ll N_1\sim N_0 \\ N_2\gtrsim N_0^{\frac{1}{d-1}}}} 
N_0^{\frac{d}{2}} \norm{\mathfrak{R}_0 v_0}_{L^4_t L^{\frac{2d}{d-1}}_x} \norm{P_{N_1} z}_{L^4_t L^{2d}_x}  \norm{Q_{>N_0^2} P_{N_2} v}_{L^2_{t,x}} \\
& \lesssim \sum _{\substack{N_2\ll N_1\sim N_0 \\ N_2\gtrsim N_0^{\frac{1}{d-1}}}} 
N_0^{\frac{d}{2}}N_0^{-1} \norm{P_{N_0}v_0}_{V^2_{\Delta}} \norm{P_{N_1} z}_{L^4_t L^{2d}_x}  \norm{P_{N_2} v}_{V^2_{\Delta}} \\
& \lesssim \sum _{\substack{N_2\ll N_1\sim N_0 \\ N_2\gtrsim N_0^{\frac{1}{d-1}}}} 
N_0^{\frac{d}{2}}N_0^{-1} N_1^{-s}N_2^{-\frac{d-2}{2}}R (N_2^{\frac{d-2}{2}}\norm{P_{N_2} v}_{V^2_{\Delta}}) \norm{P_{N_0}v_0}_{V^2_{\Delta}} \\
& \lesssim \sum _{N_1\sim N_0} 
N_1^{-s+\frac{(d-2)^2}{2(d-1)}}R\norm{v}_{Y^{\frac{d-2}{2}}} \norm{P_{N_0}v_0}_{V^2_{\Delta}} \\
& \lesssim \norm{v}_{Y^{\frac{d-2}{2}}}R
\end{align*}
for $\omega \in E_R^2$.
Here, we have used $\frac{(d-2)^2}{2(d-1)}\le \frac{(d-1)(d-2)}{2d} <s<\frac{d-2}{2}$ and $\norm{v_0}_{V^2_{\Delta}} =1$ in the last inequality.

\textbf{Subcase 3-3:} $N_1\lesssim N_2\sim N_0$.

By H\"{o}lder's inequality, Lemmas \ref{lem:Stv} and \ref{lem:Q}, we have
\begin{align*}
\left| \int _{\R ^{1+d}} P_{\lesssim N_0} z Q_{>N_0^2} P_{N_0} v \mathfrak{R}_0 v_0 dx dt \right|
& \lesssim \norm{\mathfrak{R}_0 v_0}_{L^4_t L^{\frac{2d}{d-1}}_x} \norm{P_{\lesssim N_0} z}_{L^4_t L^{2d}_x}  \norm{Q_{>N_0^2} P_{N_0} v}_{L^2_{t,x}} \\
& \lesssim N_0^{-1} \norm{P_{N_0} v_0}_{V^2_{\Delta}} \norm{z}_{L^4_t L^{2d}_x}  \norm{P_{N_0} v}_{V^2_{\Delta}} \\
& \lesssim N_0^{-1} \norm{z}_{L^4_t L^{2d}_x}  \norm{P_{N_0} v}_{V^2_{\Delta}} .
\end{align*}
Hence, we obtain
\[
\left\{ \sum _{N_0 \in 2^{\N _0}} N_0^{d} \left| \int _{\R ^{1+d}} \mathfrak{R}_0 v_0 P_{\lesssim N_0} z Q_{>N_0^2} P_{N_0} v dx dt \right| ^2 \right\} ^{\frac{1}{2}}
\lesssim R \left\{ \sum _{N_0 \in 2^{\N _0}} N_0^{d-2} \norm{P_{N_0} v}_{V^2_{\Delta}}^2 \right\} ^{\frac{1}{2}}
\lesssim \norm{v}_{Y^{\frac{d-2}{2}}} R
\]
for $\omega \in E_R^2$.
\end{proof}

\begin{rmk} \label{rmkm2}
>From \eqref{boundT}, we get the factor $T^{\delta}$ in the case 3-3 if $\omega \in E^{m,L}_R$.
In the other cases, from $\norm{Q_{>N_0^2} P_{N_0} v_0}_{L^{2}_{t,x}} \lesssim T^{\delta} \norm{Q_{>N_0^2} P_{N_0} v_0}_{L^{\frac{2}{1-2\delta}}_t L^2_x}$, we get the factor $T^{\delta}$.
\end{rmk}

\subsection{The case $m \ge 3$} \label{S:Nonestm3}

In this subsection, we consider the case $m \ge 3$.

\begin{proof}[Proof of \eqref{eq:nonest1} with $m \ge 3$]

\mbox{}

\textbf{Case 1:} $w_j=v$ ($j=1, \dots , m$) case.

This is the deterministic case and the estimate is the same as in \cite{H13}.
But, we repeat it for completeness.
>From the symmetry, we may assume that $N_1 \le \dots \le N_m$.

\textbf{Subcase 1-1:} $N_0 \sim N_m \gtrsim N_{m-1}$.

Firstly, we assume $Q_1=Q_{> N_0^2}$.
The embedding $H^{s_c}(\R ^d) \hookrightarrow L^{(m-1)d}(\R ^d)$, the $L^2$-orthogonality, and Lemma \ref{lem:Q} yield that
\begin{equation} \label{det1}
\norm{Q_{>N_0^2} P_{\lesssim N_0} v}_{L^2_t L^{(m-1)d}_x}
\lesssim \left( \sum _{N \in 2^{\N _0}, N \lesssim N_0} N^{s_c} \norm{Q_{>N_0^2} P_{N} v}_{L^2_{t,x}}^2 \right) ^{\frac{1}{2}}
\lesssim N_0^{-1} \norm{v}_{Y^{s_c}}.
\end{equation}
Similarly, we get
\begin{equation} \label{det2}
\norm{Q_{\le N_0^2} P_{\lesssim N_0} v}_{L^{\infty}_t L^{(m-1)d}_x}
\lesssim \norm{v}_{Y^{s_c}} .
\end{equation}
Accordingly, from H\"{o}lder's inequality, Lemmas \ref{lem:Stv}, \ref{lem:Q}, and above estimates, we have
\begin{align*}
& \left\{ \sum _{N_0 \in 2^{\N _0}} N_0^{2s_c+2} \left| \int _{\R ^{1+d}} Q_0 P_{N_0} v_0 Q_{>N_0^2}P_{\lesssim N_0}v \left( \prod _{j=2}^{m-1} Q_j P_{\lesssim N_0} v \right) Q_m P_{N_0} v dx dt \right| ^2 \right\} ^{\frac{1}{2}} \\
& \lesssim \Bigg\{ \sum _{N_0 \in 2^{\N _0}} N_0^{2s_c+2} \norm{Q_0 P_{N_0} v_0}_{L^4_t L^{\frac{2d}{d-1}}_x}^2 \norm{Q_{>N_0^2} P_{\lesssim N_0} v}_{L^2_t L^{(m-1)d}_x}^2 \left( \prod _{j=2}^{m-1} \norm{Q_j P_{\lesssim N_0} v}_{L^{\infty}_t L^{(m-1)d}_x}^2 \right) \\
& \hspace*{300pt} \times
\norm{Q_m P_{N_0} v}_{L^4_t L^{\frac{2d}{d-1}}_x}^2 \Bigg\} ^{\frac{1}{2}} \\
& \lesssim \left\{ \sum _{N_0 \in 2^{\N _0}} N_0^{2s_c} \norm{P_{N_0} v}_{V^2_{\Delta}}^2 \right\} ^{\frac{1}{2}} \norm{v}_{Y^{s_c}}^{m-1} \norm{v_0}_{V^2_{\Delta}} \\
& \lesssim \norm{v}_{Y^{s_c}}^m.
\end{align*}
Since the case $Q_j = Q_{> N_0^2}$ ($j=2, \dots , m-1$) follows from a similar argument as above, we omit the details.

Secondly, we consider the case where $Q_m = Q_{> N_0^2}$.
By H\"{o}lder's inequality, Bernstein's inequality, Lemma \ref{lem:Q}, and Corollary \ref{cor:biStv}, we have
\begin{align*}
& \sum _{\substack{N_1, \dots , N_{m-1} \in 2^{\N _0} \\ N_1 \le \dots \le N_{m-1} \lesssim N_0}} \left| \int _{\R ^{1+d}} \mathfrak{R}_0 v_0 \left( \prod _{j=1}^{m-1} \mathfrak{R}_j v \right) Q_{>N_0^2} P_{N_0} v dx dt \right| \\
& \lesssim \sum _{\substack{N_1, \dots , N_{m-1} \in 2^{\N _0} \\ N_1 \le \dots \le N_{m-1} \lesssim N_0}} \norm{\mathfrak{R}_0 v_0 \mathfrak{R}_{m-1} v}_{L^2_{t,x}} \left( \prod _{j=1}^{m-2} \norm{\mathfrak{R}_j v}_{L^{\infty}_{t,x}} \right) \norm{Q_{> N_0^2} P_{N_0} v}_{L^2_{t,x}} \\
& \lesssim \sum _{\substack{N_1, \dots , N_{m-1} \in 2^{\N _0} \\ N_1 \le \dots \le N_{m-1} \lesssim N_0}} N_{m-1}^{\frac{d}{2}-1} \left( \frac{N_{m-1}}{N_0} \right) ^{\frac{1}{2}-\delta} \left( \prod _{j=1}^{m-2} N_j^{\frac{d}{2}} \right) N_0^{-1} \norm{P_{N_0} v_0}_{V^2_{\Delta}} \norm{P_{N_0} v}_{V^2_{\Delta}} \prod _{j=1}^{m-1} \norm{P_{N_j} v}_{V^2_{\Delta}} \\
& \lesssim N_0^{-1} \norm{P_{N_0} v}_{V^2_{\Delta}} \norm{v}_{Y^{s_c}}^{m-1} .
\end{align*}
Hence, we obtain
\begin{align*}
& \left\{ \sum _{N_0 \in 2^{\N _0}} N_0^{2s_c+2} \left| \int _{\R ^{1+d}} Q_0 P_{N_0} v_0 \left( \prod _{j=1}^{m-1} Q_j P_{\le N_{j+1}} v \right) Q_{>N_0^2} P_{N_0} v dx dt \right| ^2 \right\} ^{\frac{1}{2}} \\
& \lesssim \left\{ \sum _{N_0 \in 2^{\N _0}} N_0^{d-1} \norm{P_{N_0} v}_{V^2_{\Delta}}^2 \right\} ^{\frac{1}{2}} \norm{v}_{Y^{s_c}}^{m-1} \\
& \lesssim \norm{v}_{Y^{s_c}}^m .
\end{align*}
We skip the proof of the case $Q_0= Q_{> N_0^2}$ because it is the same as above.

\textbf{Subcase 1-2:} $N_{m-1} \sim N_m \gtrsim N_0$.

Firstly, we assume $Q_1=Q_{> N_m^2}$.
>From H\"{o}lder's inequality, the embedding $H^{s_c}(\R ^d) \hookrightarrow L^{(m-1)d}(\R ^d)$, Lemmas \ref{lem:Stv}, \ref{lem:Q}, \eqref{det1}, and \eqref{det2}, we have
\begin{align*}
& \left| \sum _{\substack{N_{m} \in 2^{\N _0}, N_{m} \gtrsim N_0}} \int _{\R ^{1+d}} Q_0 P_{N_0} v_0 Q_{>N_m^2}P_{\le N_m}v \left( \prod _{j=2}^{m-2} Q_j P_{\le N_m} v \right) Q_{m-1} P_{N_m} v Q_{m} P_{N_m} v dx dt \right| \\
& \lesssim \sum _{\substack{N_{m} \in 2^{\N _0}, N_{m} \gtrsim N_0}} \norm{Q_0 P_{N_0} v_0}_{L^{\infty}_t L^{(m-1)d}_x} \norm{Q_{>N_m^2} P_{\le N_m} v}_{L^2_t L^{(m-1)d}_x} \left( \prod _{j=2}^{m-2} \norm{Q_j P_{\le N_m} v}_{L^{\infty}_t L^{(m-1)d}_x} \right) \\
& \hspace*{100pt} \times \norm{Q_{m-1} P_{N_m} v}_{L^4_t L^{\frac{2d}{d-1}}} \norm{Q_{m} P_{N_m} v}_{L^4_t L^{\frac{2d}{d-1}}} \\
& \lesssim \sum _{\substack{N_{m} \in 2^{\N _0}, N_{m} \gtrsim N_0}} N_0^{s_c}N_m^{-1} \norm{P_{N_0} v_0}_{V^2_{\Delta}} \norm{P_{N_m} v}_{V^2_{\Delta}}^2 \norm{v}_{Y^{s_c}}^{m-2} .
\end{align*}
Here, the part $\prod _{j=2}^{m-2} Q_j P_{\le N_m}$ disappears if $m=3$.
Hence,
\begin{align*}
& \left\{ \sum _{N_0 \in 2^{\N _0}} N_0^{2s_c+2} \left| \sum _{\substack{N_m \in 2^{\N _0} \\ N_m \gtrsim N_0}} \int _{\R ^{1+d}} Q_0 P_{N_0} v_0 Q_{>N_m^2}P_{\le N_m}v \left( \prod _{j=2}^{m-2} Q_j P_{\le N_m} \right) Q_{m-1} P_{N_m} v Q_{m} P_{N_m} v dx dt \right| ^2 \right\} ^{\frac{1}{2}} \\
& \lesssim \sum _{N_m \in 2^{\N _0}} \left\{ \sum _{N_0 \in 2^{\N _0}, N_0 \lesssim N_m} N_0^{4s_c+2} \norm{P_{N_0} v_0}_{V^2_{\Delta}} ^2 \right\} ^{\frac{1}{2}} N_m^{-1} \norm{P_{N_m}v}_{V^2_{\Delta}}^2 \norm{v}_{Y^{s_c}}^{m-2} \\
& \lesssim \sum _{N_m \in 2^{\N _0}} N_m^{2s_c} \norm{P_{N_m}v}_{V^2_{\Delta}}^2 \norm{v}_{Y^{s_c}}^{m-2} \\
& \lesssim \norm{v}_{Y^{s_c}}^m.
\end{align*}
Since the case $Q_j =Q_{> N_m^2}$ ($j=0,2,\dots ,m-2$) follows from a similar argument as above, we omit the details.

Secondly, we consider the case where $Q_m = Q_{> N_m^2}$.
>From H\"{o}lder's inequality, Corollary \ref{cor:biStv}, and Lemma \ref{lem:Q}, we have
\begin{align*}
I_{\substack{N_{m-1} \sim N_m \gtrsim N_0 \\ N_1 \le \dots \le N_m}}
& \lesssim \sum _{\substack{N_0,N_1, \dots , N_m \in 2^{\N _0} \\ N_{m-1} \sim N_m \gtrsim N_0 \\ N_1 \le \dots \le N_m}} N_0^{s_c+1} \norm{ \mathfrak{R}_0 v_0 \mathfrak{R}_{m-1}v}_{L^2_{t,x}} \left( \prod _{j=1}^{m-2} \norm{\mathfrak{R}_j v}_{L^{\infty}_{t,x}} \right) \norm{Q_{>N_m^2} P_{N_m} v}_{L^2_{t,x}} \\
& \lesssim \sum _{\substack{N_0,N_1, \dots , N_m \in 2^{\N _0} \\ N_{m-1} \sim N_m \gtrsim N_0 \\ N_1 \le \dots \le N_m}} N_0^{s_c+1} N_0^{\frac{d}{2}-1} \left( \frac{N_0}{N_{m-1}} \right) ^{\frac{1}{2}-\delta} \left( \prod _{j=1}^{m-2} N_j^{\frac{d}{2}} \right) N_m^{-1} \prod _{k=1}^m \norm{P_{N_k} v}_{V^2_{\Delta}} \\
& \lesssim \sum _{N_m \in 2^{\N _0}} N_m^{2s_c} \norm{P_{N_m} v}_{V^2_{\Delta}}^2 \norm{v}_{Y^{s_c}}^{m-2} \\
& \lesssim \norm{v}_{Y^{s_c}}^m.
\end{align*}
Similarly, the case $Q_{m-1}= Q_{> N_m^2}$ follows from the same manner.

\noindent
\textbf{Case 2:} $w_j=z$ ($j=1, \dots , m$) case.

Without loss of generality, we may assume $N_1 \le \dots \le N_m$.
Moreover, $Q_0 =Q_{> \max (N_0^2, N_m^2)}$ holds in this case.

\textbf{Subcase 2-1:} $N_0 \sim N_m$.

By H\"{o}lder's inequality, the Sobolev embedding $H^{\frac{d\delta}{2(2+\delta )}} (\R ^d) \hookrightarrow L^{2+\delta}(\R ^d)$, Lemmas \ref{lem:Q} and \ref{lem:biSt}, we get
\begin{align*}
I_{\substack{N_0 \sim N_m \\ N_1 \le \dots \le N_m}}
& \lesssim
\sum _{\substack{N_0, N_1, \dots , N_m \in 2^{\N _0} \\ N_0 \sim N_m \\ N_1 \le \dots \le N_m}}
N_0^{s_c+1} \norm{Q_{>N_0^2} P_{N_0} v_0}_{L^{2+\delta}_{t,x}} \norm{P_{N_1} z P_{N_m} z}_{L^2_{t,x}} \prod _{j=2}^{m-1} \norm{P_{N_j} z}_{L^{\frac{2(m-2)(2+\delta )}{\delta}}_{t,x}} \\
& \lesssim
\sum _{\substack{N_0, N_1, \dots , N_m \in 2^{\N _0} \\ N_0 \sim N_m \\ N_1 \le \dots \le N_m}}
N_0^{s_c+1} N_0^{-\frac{2}{2+\delta}+\frac{d\delta}{2(2+\delta )}} N_1^{\frac{d}{2}-1} \left( \frac{N_1}{N_m} \right) ^{\frac{1}{2}} \norm{P_{N_0} v_0}_{V^2_{\Delta}} \norm{P_{N_1} \phi ^{\omega}}_{L^2_x} \norm{P_{N_m} \phi ^{\omega}}_{L^2_x} \\
& \hspace*{200pt} \times \prod _{j=2}^{m-1} \norm{P_{N_j} z}_{L^{\frac{2(m-2)(2+\delta )}{\delta}}_{t,x}} \\
& \lesssim
\sum _{\substack{N_0, N_1, \dots , N_m \in 2^{\N _0} \\ N_0 \sim N_m \\ N_1 \le \dots \le N_m}}
N_0^{-s+s_c-\frac{1}{2}+\frac{(d+2)\delta}{2+\delta}} N_1^{-s+\frac{d-1}{2}} \left( \prod _{j=2}^{m-1} N_j^{-s} \right) R^{m} \\
& \lesssim
\sum _{\substack{N_0, N_1 \in 2^{\N _0} \\ N_0 \gtrsim N_1}}
N_0^{-s+s_c-\frac{1}{2}+\frac{(d+2)\delta}{2+\delta}} N_1^{-(m-1)s+\frac{d-1}{2}} R^{m} \\
& \lesssim
R^m
\end{align*}
for $\omega \in E_R^m$.
Here, we have used the fact that $s>0$ in the forth inequality and $s>\max \left( s_c-\frac{1}{2}, \frac{d}{m}-\frac{1}{m-1} \right)$ and $\delta >0$ is sufficiently small in the last inequality.
We note that this lower bound of $s$ is less than $\frac{d-1}{d} s_c$.

\textbf{Subcase 2-2:} $N_{m-1} \sim N_m \gtrsim N_0$.

By H\"{o}lder's inequality, the embedding $H^{\frac{d\delta}{2(2+\delta )}} (\R ^d) \hookrightarrow L^{2+\delta}(\R ^d)$, and Lemma \ref{lem:Q}, we get
\begin{align*}
& I_{\substack{N_{m-1} \sim N_m \gtrsim N_0 \\ N_1 \le \dots \le N_m}} \\
& \lesssim
\sum _{\substack{N_0,N_1, \dots ,N_m \in 2^{\N _0} \\ N_{m-1} \sim N_m \gtrsim N_0 \\ N_1 \le \dots \le N_m}}
N_0^{s_c+1} \norm{Q_{>N_m^2} P_{N_0} v_0}_{L^{2+\delta}_{t,x}} \left( \prod _{j=1}^{m-2} \norm{P_{N_j} z}_{L^{\frac{2(m-2)(2+\delta )}{\delta}}_{t,x}} \right) \norm{P_{N_{m-1}} z}_{L^4_{t,x}} \norm{P_{N_m} z}_{L^4_{t,x}} \\
& \lesssim
\sum _{\substack{N_0,N_1, \dots ,N_m \in 2^{\N _0} \\ N_{m-1} \sim N_m \gtrsim N_0 \\ N_1 \le \dots \le N_m}}
N_0^{s_c+1} N_m^{-\frac{2}{2+\delta}} N_0^{\frac{d\delta}{2(2+\delta )}} \norm{P_{N_0} v_0}_{V^2_{\Delta}} \left( \prod _{j=1}^{m-2} \norm{P_{N_j} z}_{L^{\frac{2(m-2)(2+\delta )}{\delta}}_{t,x}} \right) \norm{P_{N_{m-1}} z}_{L^4_{t,x}} \norm{P_{N_m} z}_{L^4_{t,x}} \\
& \lesssim
\sum _{\substack{N_0,N_1, \dots ,N_{m-1} \in 2^{\N _0} \\ N_{m-1} \gtrsim N_0 \\ N_1 \le \dots \le N_{m-1}}}
N_0^{s_c+1+\frac{d\delta}{2(2+\delta )}} N_{m-1}^{-2s-\frac{2}{2+\delta}} \left( \prod _{j=1}^{m-2} N_j^{-s} \right) R^m \\
& \lesssim
R^m
\end{align*}
for $\omega \in E_R^m$.
Here, we have used the fact that $s>\frac{s_c}{2}$, and $\delta >0$ is sufficiently small in the last inequality.

\noindent
\textbf{Case 3:} The case where there exists $l \in \{ 1, \dots ,m-1\}$ such that $w_j=v$ for $1\le j \le l$ and $w_k=z$ for $l+1 \le k \le m$.

Without loss of generality, we assume $N_1 \le \dots \le N_l$ and $N_{l+1} \le \dots \le N_m$.
We further split the proof into five subcases.

\textbf{Subcase 3-1:} $N_0 \sim N_l \gtrsim N_m$

Firstly, we assume $Q_0= Q_{>N_0^2}$ and $l=1$.
By H\"{o}lder's inequality, the Sobolev embedding $H^{\frac{d\delta}{2(2+\delta )}} (\R ^d) \hookrightarrow L^{2+\delta}(\R ^d)$, Lemma \ref{lem:Q}, and Corollary \ref{cor:biStv}, we get
\begin{align*}
I_{\substack{N_1 \sim N_0 \gtrsim N_m \\ N_2 \le \dots \le N_m}}
& \lesssim
\sum _{\substack{N_0,N_1, \dots , N_m \in 2^{\N _0} \\ N_1 \sim N_0 \gtrsim N_m \\ N_2 \le \dots \le N_m}}
N_0^{s_c+1} \norm{Q_{>N_0^2} P_{N_0} v_0}_{L^{2+\delta}_{t,x}} \norm{\mathfrak{R}_1 v \mathfrak{R}_m z}_{L^2_{t,x}} \prod _{k=2}^{m-1} \norm{\mathfrak{R}_k z}_{L^{\frac{2(m-2)(2+\delta )}{\delta}}_{t,x}} \\
& \lesssim
\sum _{\substack{N_0,N_1, \dots , N_m \in 2^{\N _0} \\ N_1 \sim N_0 \gtrsim N_m \\ N_2 \le \dots \le N_m}}
N_0^{s_c+1} N_0^{-\frac{2}{2+\delta}+\frac{d\delta}{2(2+\delta )}} N_{m}^{\frac{d}{2}-1} \left( \frac{N_m}{N_1} \right) ^{\frac{1}{2}-\delta} \norm{P_{N_0} v_0}_{V^2_{\Delta}} \norm{P_{N_1} v}_{V^2_{\Delta}} \\
& \hspace*{100pt} \times \norm{\phi ^{\omega}}_{L^2_x} \prod _{k=2}^{m-1} \norm{\mathfrak{R}_k z}_{L^{\frac{2(m-2)(2+\delta )}{\delta}}_{t,x}} \\
& \lesssim
\sum _{\substack{N_0, N_m \in 2^{\N _0} \\ N_m \lesssim N_0}}
N_0^{-\frac{1}{2}+\frac{(d+2)\delta}{2+\delta}+\delta} N_m^{-s+\frac{d-1}{2}-\delta} \norm{v}_{Y^{s_c}} R^{m-1} \\
& \lesssim
\norm{v}_{Y^{s_c}}^l R^{m-l}
\end{align*}
for $\omega \in E^m_R$.
Here, we have used the fact that $s> \frac{d}{2}-1$ and $\delta >0$ is sufficiently small in the last inequality.
We note that $\frac{d-1}{d} s_c >\frac{d}{2}-1$ if $m \ge 3$.
Since the case $Q_1= Q_{>N_0^2}$ and $l=1$ is similarly handled, we omit the details.

Secondly, we consider the case $Q_0= Q_{>N_0^2}$ and $l \ge 2$.
By H\"{o}lder's inequality, the Sobolev embedding $H^{\frac{d\delta}{2(2+\delta )}} (\R ^d) \hookrightarrow L^{2+\delta}(\R ^d)$, Lemma \ref{lem:Q}, and Corollary \ref{cor:biStv}, we get
\begin{align*}
& I_{\substack{N_l \sim N_0 \gtrsim N_m \\ N_1 \le \dots \le N_l \\ N_{l+1} \le \dots \le N_m}} \\
& \lesssim
\sum _{\substack{N_0,N_1, \dots , N_m \in 2^{\N _0} \\ N_l \sim N_0 \gtrsim N_m \\ N_1 \le \dots \le N_l \\ N_{l+1} \le \dots \le N_m}}
N_0^{s_c+1} \norm{Q_{>N_0^2} P_{N_0} v_0}_{L^{2+\delta}_{t,x}} \left( \prod _{j=1}^{l-2} \norm{\mathfrak{R}_j v}_{L^{\infty}_{t,x}} \right) \norm{\mathfrak{R}_{l-1} v \mathfrak{R}_l v}_{L^2_{t,x}} \prod _{k=l+1}^m \norm{\mathfrak{R}_k z}_{L^{\frac{2(m-l)(2+\delta )}{\delta}}_{t,x}} \\
& \lesssim
\sum _{\substack{N_0,N_1, \dots , N_m \in 2^{\N _0} \\ N_l \sim N_0 \gtrsim N_m \\ N_1 \le \dots \le N_l \\ N_{l+1} \le \dots \le N_m}}
N_0^{s_c+1} N_0^{-\frac{2}{2+\delta}+\frac{d\delta}{2(2+\delta )}} \left( \prod _{j=1}^{l-2} N_j^{\frac{d}{2}} \right) N_{l-1}^{\frac{d}{2}-1} \left( \frac{N_{l-1}}{N_l} \right) ^{\frac{1}{2}-\delta} \norm{P_{N_0} v_0}_{V^2_{\Delta}} \\
& \hspace*{100pt} \times \left( \prod _{j=1}^{l} \norm{P_{N_j} v}_{V^2_{\Delta}} \right) \prod _{k=l+1}^m \norm{\mathfrak{R}_k z}_{L^{\frac{2(m-l)(2+\delta )}{\delta}}_{t,x}} \\
& \lesssim
\sum _{\substack{N_0,N_1, \dots , N_m \in 2^{\N _0} \\ N_l \sim N_0 \gtrsim N_m \\ N_1 \le \dots \le N_l \\ N_{l+1} \le \dots \le N_m}}
N_0^{-\frac{1}{2}+\frac{(d+2) \delta}{2+\delta}+\delta} \left( \prod _{j=1}^{l-2} N_j^{\frac{1}{m-1}} \right) N_{l-1}^{\frac{1}{m-1}-\frac{1}{2}-\delta} \left( \prod _{k=l+1}^m N_k^{-s} \right)  \norm{v}_{Y^{s_c}}^l R^{m-l} \\
& \lesssim
\sum _{\substack{N_0, N_{l-1} \in 2^{\N _0} \\ N_{l-1} \lesssim N_0}}
N_0^{-\frac{1}{2}+\frac{(d+2) \delta}{2+\delta}+\delta} N_{l-1}^{\frac{l-1}{m-1}-\frac{1}{2}-\delta} \norm{v}_{Y^{s_c}}^l R^{m-l} \\
& \lesssim
\norm{v}_{Y^{s_c}}^l R^{m-l}
\end{align*}
for $\omega \in E^m_R$.
The part $\prod _{j=1}^{l-2} \norm{\mathfrak{R}_j v}_{L^{\infty}_{t,x}}$ disappears when $m=2$.
If $Q_l=Q_{>N_0^2}$, applying a similar argument as above, we obtain the desired bound.

Thirdly, we consider the case $Q_1= Q_{>N_0^2}$ and $l \ge 2$.
For $\omega \in E^m_R$ and $s>0$,
\begin{equation} \label{Zsim}
\begin{aligned}
\norm{P_{\lesssim N_0} z}_{L^{\frac{4(m-l)(4+\delta)}{\delta}}_t L^{\frac{2(m-1)(m-l)(4+\delta )d}{(m-1)(8+\delta )-2(4+\delta )(l-1)}}_x}
& \lesssim \sum _{N \in 2^{\N _0}, N \lesssim N_0} \norm{P_{N} z}_{L^{\frac{4(m-l)(4+\delta)}{\delta}}_t L^{\frac{2(m-1)(m-l)(4+\delta )d}{(m-1)(8+\delta )-2(4+\delta )(l-1)}}_x} \\
& \lesssim \sum _{N \in 2^{\N _0}, N \lesssim N_0} N^{-s} R
\lesssim R .
\end{aligned}
\end{equation}
>From H\"{o}lder's inequality, the embedding $H^{s_c}(\R ^d) \hookrightarrow L^{(m-1)d}(\R ^d)$, Lemmas \ref{lem:Stv}, \ref{lem:Q}, \eqref{det1}, \eqref{det2}, and \eqref{Zsim}, we have
\begin{align*}
& \left\{ \sum _{N_0 \in 2^{\N _0}} N_0^{2s_c+2} \left| \int _{\R ^{1+d}} Q_0 P_{N_0} v_0 Q_{>N_0^2}P_{\lesssim N_0}v \left( \prod _{j=2}^{l-1} Q_j P_{\lesssim N_0} v \right) Q_l P_{N_0} v \prod _{k=l+1}^m Q_k P_{\lesssim N_0} z dx dt \right| ^2 \right\} ^{\frac{1}{2}} \\
& \lesssim \left\{ \sum _{N_0 \in 2^{\N _0}} N_0^{2s_c+2} \norm{Q_0 P_{N_0} v_0}_{L^4_t L^{\frac{2d}{d-1}}_x}^2 \norm{Q_{>N_0^2} P_{\lesssim N_0} v}_{L^2_t L^{(m-1)d}_x}^2 \left( \prod _{j=2}^{l-1} \norm{Q_j P_{\lesssim N_0} v}_{L^{\infty}_t L^{(m-1)d}_x}^2 \right) \right. \\
& \hspace*{100pt} \left. \times  \norm{Q_l P_{N_0} v}_{L^{4+\delta}_t L^{\frac{2(4+\delta )d}{(4+\delta )d-4}}_x}^2 \prod _{k={l+1}}^m \norm{P_{\lesssim N_0} z}_{L^{\frac{4(m-l) (4+\delta )}{\delta}}_t L^{\frac{2(m-1)(m-l)(4+\delta )d}{(m-1)(8+\delta )-2(4+\delta )(l-1)}}_x} \right\} ^{\frac{1}{2}} \\
& \lesssim \left\{ \sum _{N_0 \in 2^{\N _0}} N_0^{2s_c} \norm{P_{N_0} v}_{V^2_{\Delta}}^2 \right\} ^{\frac{1}{2}} \norm{v}_{Y^{s_c}}^{l-1} R^{m-l} \\
& \lesssim \norm{v}_{Y^{s_c}}^{l} R^{m-l}
\end{align*}
for $\omega \in E^m_R$.
The proof of the remaining cases $Q_j= Q_{>N_0^2}$ ($j=2, \dots , l-1$) follows from the same manner.

\textbf{Subcase 3-2:} $N_0 \sim N_m \gtrsim N_l$.

The lower bound $\max \left( \frac{d-1}{d} s_c, s_c - \frac{d}{2(d+1)} \right)$ of Lemma \ref{lem:nonest} appears in this case.
We further divide the proof into two subcases.

\textbf{Subsubcase 3-2-1:} $N_0 \sim N_m \gtrsim N_l$ and $Q_0 = Q_{>N_0^2}$.

When $l \le m-2$, from H\"{o}lder's inequality, the embedding $H^{\frac{d\delta}{2(2+\delta )}} (\R ^d) \hookrightarrow L^{2+\delta}(\R ^d)$, Lemma \ref{lem:Q},  Corollary \ref{cor:biStv}, and $m \ge 3$, we have
\begin{align*}
& I_{\substack{N_0 \sim N_m \gtrsim N_l \\ N_1 \le \dots \le N_l \\ N_{l+1} \le \dots \le N_m}} \\
& \lesssim
\sum _{\substack{N_0,N_1, \dots , N_m \in 2^{\N _0} \\ N_0 \sim N_m \gtrsim N_l \\ N_1 \le \dots \le N_l \\ N_{l+1} \le \dots \le N_m}}
N_0^{s_c+1} \norm{Q_{>N_0^2} P_{N_0} v_0}_{L^{2+\delta}_{t,x}} \left( \prod _{j=1}^{l-1} \norm{\mathfrak{R}_j v}_{L^{\infty}_{t,x}} \right) \norm{\mathfrak{R}_l v \mathfrak{R}_{m} z }_{L^2_{t,x}} \prod _{k=l+1}^{m-1} \norm{\mathfrak{R}_k z}_{L^{\frac{2(m-l-1)(2+\delta )}{\delta}}_{t,x}} \\
& \lesssim
\sum _{\substack{N_0,N_1, \dots , N_m \in 2^{\N _0} \\ N_0 \sim N_m \gtrsim N_l \\ N_1 \le \dots \le N_l \\ N_{l+1} \le \dots \le N_m}}
N_0^{s_c+1} N_0^{-\frac{2}{2+\delta}+\frac{d\delta}{2(2+\delta )}} \left( \prod _{j=1}^{l-1} N_j^{\frac{d}{2}} \right) N_l^{\frac{d}{2}-1} \left( \frac{N_l}{N_m} \right) ^{\frac{1}{2}-\delta} \norm{P_{N_0} v_0}_{V^2_{\Delta}} \\
& \hspace*{50pt} \times \left( \prod _{j=1}^l \norm{P_{N_j} v}_{V^2_{\Delta}} \right) \norm{\phi ^{\omega}}_{L^2_x} \prod _{k=l+1}^{m-1} \norm{\mathfrak{R}_k z}_{L^{\frac{2(m-l-1)(2+\delta )}{\delta}}_{t,x}} \\
& \lesssim
\sum _{\substack{N_0,N_1, \dots , N_l \in 2^{\N _0} \\ N_1 \le \dots \le N_l \lesssim N_0}}
N_0^{-s+s_c-\frac{1}{2}+\frac{(d+2) \delta}{2+\delta}+\delta} \left( \prod _{j=1}^{l-1} N_j^{\frac{1}{m-1}} \right) N_l^{\frac{1}{m-1}-\frac{1}{2}-\delta} \norm{v}_{Y^{s_c}}^l R^{m-l} \\
& \lesssim
\sum _{\substack{N_0,N_{l-1}, N_l \in 2^{\N _0} \\ N_{l-1} \le N_l \lesssim N_0}}
N_0^{-s+s_c-\frac{1}{2}+\frac{(d+2)\delta}{2+\delta}+\delta} N_{l-1}^{\frac{l-1}{m-1}} N_l^{\frac{1}{m-1}-\frac{1}{2}-\delta} \norm{v}_{Y^{s_c}}^l R^{m-l} \\
& \lesssim
\sum _{\substack{N_0,N_{l-1} \in 2^{\N _0} \\ N_{l-1} \lesssim N_0}}
N_0^{-s+s_c-\frac{1}{2}+\frac{(d+2)\delta}{2+\delta}+\delta} N_{l-1}^{\frac{l}{m-1}-\frac{1}{2}-\delta} \norm{v}_{Y^{s_c}}^l R^{m-l}
\end{align*}
for $\omega \in E^m_R$.
When $l=1$, the part $\prod _{j=1}^{l-1} \norm{\mathfrak{R}_j v}_{L^{\infty}_{t,x}}$ disappears.
If $l=m-1$, replacing $\norm{Q_{>N_0^2} P_{N_0} v_0}_{L^{2+\delta}_{t,x}}$ with $\norm{Q_{>N_0^2} P_{N_0} v_0}_{L^2_{t,x}}$, we get the same bound as above.
We note that $\prod _{k=l+1}^{m-1} \norm{\mathfrak{R}_j z}_{L^{\frac{2(m-l-1)(2+\delta )}{\delta}}_{t,x}}$ disappears in this case.

The sum is bounded by $\norm{v}_{Y^{s_c}} R^{m-1}$ if $l=1$, $m \ge 3$, $s>s_c-\frac{1}{2}$, and $\delta >0$ is sufficiently small.
Moreover, if $l \ge 2$, $N_{l-1} \lesssim N_0^{\frac{1}{d+1}}$ and $s> \max \left( s_c-\frac{1}{2}, s_c-\frac{d}{2(d+1)} \right)$, we get the same bound.

We consider the case $l \ge 2$ and $N_{l-1} \gtrsim N_0^{\frac{1}{d+1}}$.
>From H\"{o}lder's inequality, the embeddings $W^{\frac{d}{2}-1, \frac{2d}{d-1}} (\R ^d) \hookrightarrow L^{2d}(\R ^d)$, $H^{\frac{d\delta}{2(2+\delta )}} (\R ^d) \hookrightarrow L^{2+\delta}(\R ^d)$, Lemmas \ref{lem:Stv}, and \ref{lem:Q}, we have
\begin{align*}
& I_{\substack{N_0 \sim N_m \gtrsim N_l \\ N_{l-1} \gtrsim N_0^{\frac{1}{d+1}} \\ N_1 \le \dots \le N_l \\ N_{l+1} \le \dots \le N_m}} \\
& \lesssim
\sum _{\substack{N_0,N_1, \dots , N_m \in 2^{\N _0} \\ N_0 \sim N_m \gtrsim N_l \\ N_{l-1} \gtrsim N_0^{\frac{1}{d+1}} \\ N_1 \le \dots \le N_l \\ N_{l+1} \le \dots \le N_m}}
N_0^{s_c+1} \norm{Q_{>N_0^2} P_{N_0} v_0}_{L^{2+\delta}_{t,x}} \left( \prod _{j=1}^{l-2} \norm{\mathfrak{R}_j v}_{L^{\infty}_{t,x}} \right) \norm{\mathfrak{R}_{l-1} v}_{L^4_t L^{\frac{2d}{d-1}}_x} \norm{\mathfrak{R}_l v}_{L^4_t L^{2d}_x} \\
& \hspace*{100pt} \times \prod _{k=l+1}^m \norm{\mathfrak{R}_k z}_{L^{\frac{2(m-l)(2+\delta )}{\delta}}_{t,x}} \\
& \lesssim
\sum _{\substack{N_0,N_1, \dots , N_m \in 2^{\N _0} \\ N_0 \sim N_m \gtrsim N_l \\ N_{l-1} \gtrsim N_0^{\frac{1}{d+1}} \\ N_1 \le \dots \le N_l \\ N_{l+1} \le \dots \le N_m}}
N_0^{s_c+1} N_0^{-\frac{2}{2+\delta}+\frac{d\delta}{2(2+\delta )}} \left( \prod _{j=1}^{l-2} N_j^{\frac{d}{2}} \right) N_l^{\frac{d}{2}-1} \left( \prod _{j=1}^l \norm{P_{N_j} v}_{V^2_{\Delta}} \right) \prod _{k=l+1}^m \norm{\mathfrak{R}_k z}_{L^{\frac{2(m-l)(2+\delta )}{\delta}}_{t,x}} \\
& \lesssim
\sum _{\substack{N_0,N_1, \dots , N_l \in 2^{\N _0} \\ N_{l-1} \gtrsim N_0^{\frac{1}{d+1}} \\ N_1 \le \dots \le N_l \lesssim N_0}}
N_0^{-s+s_c+\frac{(d+2)\delta}{2(2+\delta )}} \left( \prod _{j=1}^{l-2} N_j^{\frac{1}{m-1}} \right) N_{l-1}^{-s_c} N_l^{\frac{1}{m-1}-1} \norm{v}_{Y^{s_c}}^l R^{m-l} \\
& \lesssim
\sum _{\substack{N_0N_{l-1} \in 2^{\N _0} \\ N_{l-1} \gtrsim N_0^{\frac{1}{d+1}}}}
N_0^{-s+s_c+\frac{(d+2)\delta}{2(2+\delta )}} N_{l-1}^{-s_c-\frac{m-l}{m-1}} \norm{v}_{Y^{s_c}}^l R^{m-l} \\
& \lesssim
\norm{v}_{Y^{s_c}}^l R^{m-l}
\end{align*}
for $\omega \in E^m_R$.
Here, we have used the fact that $s>s_c-\frac{d}{2(d+1)}$ and $\delta >0$ is sufficiently small in the last inequality.

\textbf{Subsubcase 3-2-2:} $N_0 \sim N_m \gtrsim N_l$ and $Q_0 = Q_{\ll N_0^2}$.

We only consider the case $Q_1= Q_{>N_0^2}$ because the remaining cases are similarly handled.
Firstly, we consider the case $l \ge 3$.
>From H\"{o}lder's inequality, the embeddings $H^{s_c}(\R ^d) \hookrightarrow L^{(m-1)d}(\R ^d)$, $H^{\frac{d}{2}-\frac{(1-2\delta )(m-1)d-2}{2(m-1)}} (\R ^d) \hookrightarrow L^{\frac{2(m-1)d}{(1-2\delta )(m-1)d-2}}(\R ^d)$, Lemma \ref{lem:Q}, and Corollary \ref{cor:biStv} we have
\begin{align*}
I_{\substack{N_0 \sim N_m \gtrsim N_l \\ N_1 \le \dots \le N_l \\ N_{l+1} \le \dots \le N_m}}
& \lesssim
\sum _{\substack{N_0,N_1, \dots , N_m \in 2^{\N _0} \\ N_0 \sim N_m \gtrsim N_l \\ N_1 \le \dots \le N_l \\ N_{l+1} \le \dots \le N_m}}
N_0^{s_c+1} \norm{\mathfrak{R}_0 v_0 \mathfrak{R}_{l} v}_{L^2_{t,x}} \norm{Q_{>N_0^2} P_{N_1} v}_{L^{2+\delta}_t L^{(m-1)d}_x} \left( \prod _{j=2}^{l-2} \norm{\mathfrak{R}_j v}_{L^{\infty}_{t,x}} \right) \\
& \hspace*{50pt} \times \norm{P_{N_{l-1}} v}_{L^{\infty}_t L^{\frac{2(m-1)d}{(1-2\delta )(m-1)d-2}}_x} \left( \prod _{k=l+1}^{m} \norm{\mathfrak{R}_k z}_{L^{\frac{2(m-l)(2+\delta )}{\delta}}_t L^{\frac{m-l}{\delta}}_x} \right) \\
& \lesssim
\sum _{\substack{N_0,N_1, \dots , N_m \in 2^{\N _0} \\ N_0 \sim N_m \gtrsim N_l \\ N_1 \le \dots \le N_l \\ N_{l+1} \le \dots \le N_m}}
N_0^{s_c+1} N_l^{\frac{d}{2}-1} \left( \frac{N_l}{N_0} \right) ^{\frac{1}{2}-\delta} N_0^{-\frac{2}{2+\delta}} N_1^{s_c} \left( \prod _{j=2}^{l-2} N_j^{\frac{d}{2}} \right) N_{l-1}^{\frac{d}{2}-\frac{(1-2\delta )(m-1)d-2}{2(m-1)}} \\
& \hspace*{50pt} \times \norm{P_{N_0} v_0}_{V^2_{\Delta}} \left( \prod _{j=1}^l \norm{P_{N_j} v}_{V^2_{\Delta}} \right) \left( \prod _{k=l+1}^{m} \norm{\mathfrak{R}_k z}_{L^{\frac{2(m-l)(2+\delta )}{\delta}}_t L^{\frac{m-l}{\delta}}_x} \right) \\
& \lesssim
\sum _{\substack{N_0,N_1, \dots , N_m \in 2^{\N _0} \\ N_0 \sim N_m \gtrsim N_l \\ N_1 \le \dots \le N_l \\ N_{l+1} \le \dots \le N_m}}
N_0^{-s+s_c-\frac{1}{2}+\frac{\delta}{2+\delta}+2\delta} N_l^{\frac{1}{m-1}-\frac{1}{2}-\delta}  \left( \prod _{j=2}^{l-2} N_j^{\frac{1}{m-1}} \right) N_{l-1}^{-s_c +\frac{1}{m-1}+d \delta } \norm{v}_{Y^{s_c}}^l R^{m-l} \\
& \lesssim
\sum _{\substack{N_0,N_{l-1} \in 2^{\N _0} \\ N_0 \gtrsim N_{l-1}}}
N_0^{-s+s_c-\frac{1}{2}+\frac{\delta}{2+\delta}+2\delta} N_{l-1}^{-s_c-\frac{1}{2} +\frac{l-1}{m-1}+(d-1) \delta } \norm{v}_{Y^{s_c}}^l R^{m-l} \\
& \lesssim
\norm{v}_{Y^{s_c}}^l R^{m-l}
\end{align*}
for $\omega \in E^m_R$.
Here, we have used the fact that $s>s_c-\frac{1}{2}$ and $\delta >0$ is sufficiently small in the last inequality.

Secondly, we consider the case $l=2$.
>From H\"{o}lder's inequality, the embedding $H^{\frac{d\delta}{2(2+\delta )}} (\R ^d) \hookrightarrow L^{2+\delta}(\R ^d)$, Lemma \ref{lem:Q}, and Corollary \ref{cor:biStv} we have
\begin{align*}
I_{\substack{N_0 \sim N_m \gtrsim N_2 \\ N_1 \le N_2 \\ N_3 \le \dots \le N_m}}
& \lesssim
\sum _{\substack{N_0,N_1, \dots , N_m \in 2^{\N _0} \\ N_0 \sim N_m \gtrsim N_2 \\ N_1 \le N_2 \\ N_3 \le \dots \le N_m}}
N_0^{s_c+1} \norm{\mathfrak{R}_0 v_0 \mathfrak{R}_2 v}_{L^2_{t,x}} \norm{Q_{>N_0^2} P_{N_1} v}_{L^{2+\delta}_{t,x}} \prod _{k=3}^m \norm{\mathfrak{R}_k z}_{L^{\frac{2(m-2)(2+\delta )}{\delta}}_{t,x}} \\
& \lesssim
\sum _{\substack{N_0,N_1, \dots , N_m \in 2^{\N _0} \\ N_0 \sim N_m \gtrsim N_2 \\ N_1 \le N_2 \\ N_3 \le \dots \le N_m}}
N_0^{s_c+1} N_2^{\frac{d}{2}-1} \left( \frac{N_2}{N_0} \right) ^{\frac{1}{2}-\delta} N_0^{-\frac{2}{2+\delta}} N_1^{\frac{d\delta}{2(2+\delta )}} \norm{P_{N_0} v_0}_{V^2_{\Delta}} \left( \prod _{j=1}^2 \norm{P_{N_j} v}_{V^2_{\Delta}} \right) \\
& \hspace*{200pt} \times \prod _{k=3}^{m} \norm{\mathfrak{R}_k z}_{L^{\frac{2(m-l)(2+\delta )}{\delta}}_{t,x}} \\
& \lesssim
\sum _{\substack{N_0,N_1,N_2 \in 2^{\N _0} \\ N_1 \le N_2 \lesssim N_0}}
N_0^{-s+s_c-\frac{1}{2}+\frac{\delta}{2+\delta}+\delta} N_2^{\frac{1}{m-1}-\frac{1}{2}-\delta} N_1^{-s_c+\frac{d\delta}{2(2+\delta )}} \norm{v}_{Y^{s_c}}^2 R^{m-2} \\
& \lesssim
\norm{v}_{Y^{s_c}}^2 R^{m-2}
\end{align*}
for $\omega \in E^m_R$.
Here, we have used the fact that $m \ge 3$, $s>s_c-\frac{1}{2}$, and $\delta >0$ is sufficiently small in the last inequality.

Thirdly, we assume $l=1$ and $N_{m-1} \gtrsim N_0^{\frac{1}{d-1}}$.
>From H\"{o}lder's inequality, the embeddings $H^{s_c}(\R ^d) \hookrightarrow L^{(m-1)d}(\R ^d)$, $H^{\frac{d\delta}{2(2+\delta )}} (\R ^d) \hookrightarrow L^{2+\delta}(\R ^d)$, Lemmas \ref{lem:Stv}, and \ref{lem:Q}, we have
\begin{align*}
I_{\substack{N_0 \sim N_m \gtrsim N_1 \\ N_{m-1} \gtrsim N_0^{\frac{1}{d-1}} \\ N_2 \le \dots \le N_m}}
& \lesssim
\sum _{\substack{N_0,N_1, \dots , N_m \in 2^{\N _0} \\ N_0 \sim N_m \gtrsim N_1 \\ N_{m-1} \gtrsim N_0^{\frac{1}{d-1}} \\ N_2 \le \dots \le N_m}}
N_0^{s_c+1} \norm{\mathfrak{R}_0 v_0}_{L^4_t L^{\frac{2d}{d-1}}} \norm{Q_{>N_0^2} P_{N_1} v}_{L^2_t L^{(m-1)d}_x} \prod _{k=2}^m \norm{\mathfrak{R}_k z}_{L^{4(m-1)}_t L^{\frac{2(m-1)^2d}{(m-1)d+m-3}}_x} \\
& \lesssim
\sum _{\substack{N_0,N_1, \dots , N_m \in 2^{\N _0} \\ N_0 \sim N_m \gtrsim N_1 \\ N_{m-1} \gtrsim N_0^{\frac{1}{d-1}} \\ N_2 \le \dots \le N_m}}
N_0^{s_c+1} N_0^{-1} N_1^{s_c} \norm{P_{N_0} v_0}_{V^2_{\Delta}} \norm{P_{N_1} v}_{V^2_{\Delta}} \prod _{k=2}^m \norm{\mathfrak{R}_k z}_{L^{4(m-1)}_t L^{\frac{2(m-1)^2d}{(m-1)d+m-3}}_x} \\
& \lesssim
\sum _{\substack{N_0,N_{m-1} \in 2^{\N _0} \\ N_{m-1} \gtrsim N_0^{\frac{1}{d-1}}}}
N_0^{-s+s_c+\delta} N_{m-1}^{-s} \norm{v}_{Y^{s_c}} R^{m-1} \\
& \lesssim
\sum _{\substack{N_0 \in 2^{\N _0}}}
N_0^{-\frac{d}{d-1}s+s_c+\delta} \norm{v}_{Y^{s_c}} R^{m-1} \\
& \lesssim
\norm{v}_{Y^{s_c}} R^{m-1}
\end{align*}
for $\omega \in E^m_R$.
Here, we have used the fact that $s> \frac{d-1}{d}s_c$ and $\delta >0$ is sufficiently small in the last inequality.

Fourthly, we consider the case $l=1$ and $N_{m-1} \lesssim N_0^{\frac{1}{d-1}}$.
>From H\"{o}lder's inequality, the embedding $H^{s_c}(\R ^d) \hookrightarrow L^{(m-1)d}(\R ^d)$, Lemma \ref{lem:Q}, and Corollary \ref{cor:biStv} we have
\begin{align*}
I_{\substack{N_0 \sim N_m \gtrsim N_1 \\ N_{m-1} \lesssim N_0^{\frac{1}{d-1}} \\ N_2 \le \dots \le N_m}}
& \lesssim
\sum _{\substack{N_0,N_1, \dots , N_m \in 2^{\N _0} \\ N_0 \sim N_m \gtrsim N_1 \\ N_{m-1} \lesssim N_0^{\frac{1}{d-1}} \\ N_2 \le \dots \le N_m}}
N_0^{s_c+1} \norm{\mathfrak{R}_0 v_0 \mathfrak{R}_{m-1} z}_{L^2_{t,x}} \norm{Q_{>N_0^2} P_{N_1} v}_{L^{2+\delta}_t L^{(m-1)d}_x} \\
& \hspace*{50pt} \times \left( \prod _{k=2}^{m-2} \norm{\mathfrak{R}_k z}_{L^{\frac{2(m-2)(2+\delta )}{\delta}}_t L^{\frac{2(m-1)(m-2)d}{(m-1)d-2}}_x} \right) \norm{\mathfrak{R}_m z}_{L^{\frac{2(m-2)(2+\delta )}{\delta}}_t L^{\frac{2(m-1)(m-2)d}{(m-1)d-2}}_x} \\
& \lesssim
\sum _{\substack{N_0,N_1, \dots , N_m \in 2^{\N _0} \\ N_0 \sim N_m \gtrsim N_1 \\ N_{m-1} \lesssim N_0^{\frac{1}{d-1}} \\ N_2 \le \dots \le N_m}}
N_0^{s_c+1} N_{m-1}^{\frac{d}{2}-1} \left( \frac{N_{m-1}}{N_0} \right) ^{\frac{1}{2}-\delta} N_0^{-\frac{2}{2+\delta}} N_1^{s_c} \norm{P_{N_0} v_0}_{V^2_{\Delta}} \norm{P_{N_1} v}_{V^2_{\Delta}} \\
& \hspace*{30pt} \times \norm{\phi ^{\omega}}_{L^2_x} \left( \prod _{k=2}^{m-2} \norm{\mathfrak{R}_k z}_{L^{\frac{2(m-2)(2+\delta )}{\delta}}_t L^{\frac{2(m-1)(m-2)d}{(m-1)d-2}}_x} \right) \norm{\mathfrak{R}_m z}_{L^{\frac{2(m-2)(2+\delta )}{\delta}}_t L^{\frac{2(m-1)(m-2)d}{(m-1)d-2}}_x}  \\
& \lesssim
\sum _{\substack{N_0,N_{m-1} \in 2^{\N _0} \\ N_{m-1} \lesssim N_0^{\frac{1}{d-1}}}}
N_0^{-s+s_c-\frac{1}{2}+\frac{\delta}{2+\delta}+2\delta} N_{m-1}^{-s+\frac{d-1}{2}-\delta} \norm{v}_{Y^{s_c}} R^{m-1} \\
& \lesssim
\norm{v}_{Y^{s_c}} R^{m-1}
\end{align*}
for $\omega \in E^m_R$.
When $m=3$, the part $\prod _{k=2}^{m-2} \norm{\mathfrak{R}_k z}_{L^{\frac{2(m-2)(2+\delta )}{\delta}}_t L^{\frac{2(m-1)(m-2)d}{(m-1)d-2}}_x}$ disappears.
Here, we have used the fact that $s>\max \left( s_c-\frac{1}{2}, \frac{d-1}{d}s_c \right)$ and $\delta >0$ is sufficiently small in the last inequality.

\textbf{Subcase 3-3:} $N_{l-1} \sim N_l \gtrsim N_0, N_m$

We assume $l \ge 2$ because this is reduced the subcase 3-1 when $l=1$.

Firstly, we consider the case $Q_0=Q_{>N_l^2}$.
>From H\"{o}lder's inequality, the embedding $H^{s_c}(\R ^d) \hookrightarrow L^{(m-1)d}(\R ^d)$, Lemmas \ref{lem:Stv}, \ref{lem:Q}, \eqref{det1}, \eqref{det2}, and \eqref{Zsim}, we have
\begin{align*}
& \left| \int _{\R ^{1+d}} Q_{>N_l^2} P_{\lesssim N_l}v \left( \prod _{j=1}^{l-2} Q_j P_{\lesssim N_l} v \right) Q_{l-1} P_{N_l} v Q_l P_{N_l} v \left( \prod _{k=l+1}^m Q_k P_{\lesssim N_l} z \right) dx dt \right| \\
& \lesssim
\norm{Q_{>N_l^2} P_{N_0} v_0}_{L^2_t L^{(m-1)d}_x} \left( \prod _{j=1}^{l-2} \norm{Q_j P_{\lesssim N_l} v}_{L^{\infty}_t L^{(m-1)d}_x} \right) \\
& \hspace*{30pt} \times \norm{Q_{l-1} P_{N_l} v}_{L^4_t L^{\frac{2d}{d-1}}_x} \norm{Q_l P_{N_l} v}_{L^{4+\delta}_t L^{\frac{2(4+\delta )d}{(4+\delta )d-4}}_x} \prod _{k=l+1}^m \norm{Q_k P_{\lesssim N_l} z}_{L^{\frac{4(m-l)(4+\delta )}{\delta}}_t L^{\frac{2(m-1)(m-l)(4+\delta )d}{(m-1)(8+\delta )-2(4+\delta )(l-1)}}_x} \\
& \lesssim
N_0^{s_c} N_l^{-1} \norm{P_{N_0} v_0}_{V^2_{\Delta}} \norm{P_{N_l} v}_{V^2_{\Delta}}^2 \norm{v}_{Y^{s_c}}^{l-2} R^{m-l}
\end{align*}
for $\omega \in E^m_R$.
Accordingly, we obtain
\begin{align*}
& \left\{ \sum _{N_0 \in 2^{\N _0}} N_0^{2s_c+2} \left| \sum _{\substack{N_l \in 2^{\N _0} \\ N_0 \lesssim N_l}} \int _{\R ^{1+d}} Q_0 P_{N_0} v_0 Q_{>N_l^2}P_{\lesssim N_l}v \left( \prod _{j=2}^l Q_j P_{\lesssim N_l} v \right) \left( \prod _{k=l+1}^m Q_k P_{\lesssim N_l} z \right) dx dt \right| ^2 \right\} ^{\frac{1}{2}} \\
& \lesssim
\sum _{N_l \in 2^{\N _0}} \left\{ \sum _{N_0 \in 2^{\N _0}} N_0^{4s_c+2} \right\} ^{\frac{1}{2}} N_l^{-1} \norm{P_{N_l} v}_{V^2_{\Delta}}^2 \norm{v}_{Y^{s_c}}^{l-2} R^{m-l} \\
& \lesssim
\sum _{N_l \in 2^{\N _0}} N_l^{2s_c} \norm{P_{N_l} v}_{V^2_{\Delta}}^2 \norm{v}_{Y^{s_c}}^{l-2} R^{m-l} \\
& \lesssim \norm{v}_{Y^{s_c}}^l R^{m-l}
\end{align*}
for $\omega \in E^m_R$.
Since the case where $Q_j=Q_{>N_l^2}$ for some $j=1, \dots , l-2$ is similarly handled, we omit the detail.

Secondly, we consider the case $Q_l = Q_{>N_l^2}$.
By H\"{o}lder's inequality, the Sobolev embedding $H^{\frac{d\delta}{2(2+\delta )}} (\R ^d) \hookrightarrow L^{2+\delta}(\R ^d)$, Lemma \ref{lem:Q}, and Corollary \ref{cor:biStv}, we get
\begin{align*}
& I_{\substack{N_{l-1} \sim N_l \gtrsim N_0, N_m \\ N_1 \le \dots \le N_l \\ N_{l+1} \le \dots \le N_m}} \\
& \lesssim
\sum _{\substack{N_0,N_1, \dots , N_m \in 2^{\N _0} \\ N_{l-1} \sim N_l \gtrsim N_0, N_m \\ N_1 \le \dots \le N_l \\ N_{l+1} \le \dots \le N_m}}
N_0^{s_c+1} \norm{\mathfrak{R}_0 v_0 \mathfrak{R}_{l-1} v}_{L^2_{t,x}} \left( \prod _{j=1}^{l-2} \norm{\mathfrak{R}_j v}_{L^{\infty}_{t,x}} \right) \norm{Q_{>N_l^2} P_{N_l} v}_{L^{2+\delta}_{t,x}} \prod _{k=l+1}^m \norm{\mathfrak{R}_k z}_{L^{\frac{2(m-l)(2+\delta )}{\delta}}_{t,x}} \\
& \lesssim
\sum _{\substack{N_0,N_1, \dots , N_m \in 2^{\N _0} \\ N_{l-1} \sim N_l \gtrsim N_0, N_m \\ N_1 \le \dots \le N_l \\ N_{l+1} \le \dots \le N_m}}
N_0^{s_c+1} N_0^{\frac{d}{2}-1} \left( \frac{N_0}{N_{l-1}} \right) ^{\frac{1}{2}-\delta} \left( \prod _{j=1}^{l-2} N_j^{\frac{d}{2}} \right) N_l^{-\frac{2}{2+\delta}+\frac{d\delta}{2(2+\delta )}} \norm{P_{N_0} v_0}_{V^2_{\Delta}} \\
& \hspace*{100pt} \times \left( \prod _{j=1}^l \norm{P_{N_j} v}_{V^2_{\Delta}} \right) \prod _{k=l+1}^m \norm{\mathfrak{R}_k z}_{L^{\frac{2(m-l)(2+\delta )}{\delta}}_{t,x}} \\
& \lesssim
\sum _{\substack{N_0,N_1, \dots , N_l \in 2^{\N _0} \\ N_{l-1} \sim N_l \gtrsim N_0, N_m \\ N_1 \le \dots \le N_l}}
N_0^{2s_c+\frac{1}{2}+\frac{1}{m-1}-\delta} \left( \prod _{j=1}^{l-2} N_j^{\frac{1}{m-1}} \right) N_l^{-2s_c-\frac{1}{2}-\frac{2}{2+\delta}+\frac{d\delta}{2(2+\delta )}+\delta} \norm{v}_{Y^{s_c}}^l R^{m-l} \\
& \lesssim
\sum _{N_l \in 2^{\N _0}}
N_l^{-\frac{m-l}{m-1}+\frac{(d+2)\delta}{2+\delta}} \norm{v}_{Y^{s_c}}^l R^{m-l} \\
& \lesssim
\norm{v}_{Y^{s_c}}^l R^{m-l}
\end{align*}
for $\omega \in E^m_R$.
Here, we have used the fact that $1 \le l \le m-1$ and $\delta >0$ is sufficiently small in the last inequality.

\textbf{Subcase 3-4:} $N_l \sim N_m \gtrsim N_0$

We further divide the proof into two subcases.

\textbf{Subsubcase 3-4-1:} $l=1$

Firstly, we consider the case $Q_0=Q_{>N_1^2}$.
By H\"{o}lder's inequality, the Sobolev embedding $H^{\frac{d\delta}{2(2+\delta )}} (\R ^d) \hookrightarrow L^{2+\delta}(\R ^d)$, Lemma \ref{lem:Q}, and Corollary \ref{cor:biStv}, we get
\begin{align*}
I_{\substack{N_1 \sim N_m \gtrsim N_0 \\ N_2 \le \dots \le N_m}}
& \lesssim
\sum _{\substack{N_0,N_1, \dots , N_m \in 2^{\N _0} \\ N_1 \sim N_m \gtrsim N_0 \\ N_2 \le \dots \le N_m}}
N_0^{s_c+1} \norm{Q_{>N_1^2} P_{N_0} v_0}_{L^{2+\delta}_{t,x}} \norm{\mathfrak{R}_1 v \mathfrak{R}_2z}_{L^2_{t,x}} \prod _{k=3}^m \norm{\mathfrak{R}_k z}_{L^{\frac{2(m-2)(2+\delta )}{\delta}}_{t,x}} \\
& \lesssim
\sum _{\substack{N_0,N_1, \dots , N_m \in 2^{\N _0} \\ N_1 \sim N_m \gtrsim N_0 \\ N_2 \le \dots \le N_m}}
N_0^{s_c+1} N_1^{-\frac{2}{2+\delta}} N_0^{\frac{d\delta}{2(2+\delta )}} N_2 ^{\frac{d}{2}-1} \left( \frac{N_2}{N_1} \right) ^{\frac{1}{2}-\delta} \norm{P_{N_0} v_0}_{V^2_{\Delta}} \norm{P_{N_1} v}_{V^2_{\Delta}} \norm{\phi ^{\omega}}_{L^2_x} \\
& \hspace*{200pt} \times \prod _{k=3}^m \norm{\mathfrak{R}_k z}_{L^{\frac{2(m-2)(2+\delta )}{\delta}}_{t,x}} \\
& \lesssim
\sum _{\substack{N_0,N_1,N_2 \in 2^{\N _0} \\ N_1 \gtrsim N_0, N_2}}
N_0^{s_c+1+\frac{d\delta}{2(2+\delta )}} N_1^{-s-s_c-\frac{1}{2}-\frac{2}{2+\delta}+\delta} N_2 ^{-s+\frac{d-1}{2}-\delta} \norm{v}_{Y^{s_c}} R^{m-1} \\
& \lesssim
\sum _{\substack{N_1,N_2 \in 2^{\N _0} \\ N_1 \gtrsim N_2}}
N_1^{-s-\frac{1}{2}+\frac{(d+2)\delta}{2+\delta}+\delta} N_2 ^{-s+\frac{d-1}{2}-\delta} \norm{v}_{Y^{s_c}} R^{m-1} \\
& \lesssim
\norm{v}_{Y^{s_c}} R^{m-1}
\end{align*}
for $\omega \in E^m_R$.
Here, we have used the fact that $s> \frac{1}{2}( \frac{d}{2}-1)$ and $\delta >0$ is sufficiently small in the last inequality.

Secondly, we consider the case $Q_1=Q_{>N_1^2}$.
By H\"{o}lder's inequality, the Sobolev embedding $H^{\frac{d\delta}{2(2+\delta )}} (\R ^d) \hookrightarrow L^{2+\delta}(\R ^d)$, Lemma \ref{lem:Q}, and Corollary \ref{cor:biStv}, we get
\begin{align*}
I_{\substack{N_1 \sim N_m \gtrsim N_0 \\ N_2 \le \dots \le N_m}}
& \lesssim
\sum _{\substack{N_0,N_1, \dots , N_m \in 2^{\N _0} \\ N_1 \sim N_m \gtrsim N_0 \\ N_2 \le \dots \le N_m}}
N_0^{s_c+1} \norm{\mathfrak{R}_0 v \mathfrak{R}_m z}_{L^2_{t,x}} \norm{Q_{>N_1^2} P_{N_1} v}_{L^{2+\delta}_{t,x}} \prod _{k=2}^{m-1} \norm{\mathfrak{R}_k z}_{L^{\frac{2(m-2)(2+\delta )}{\delta}}_{t,x}} \\
& \lesssim
\sum _{\substack{N_0,N_1, \dots , N_m \in 2^{\N _0} \\ N_1 \sim N_m \gtrsim N_0 \\ N_2 \le \dots \le N_m}}
N_0^{s_c+1} N_0^{\frac{d}{2}-1} \left( \frac{N_0}{N_m} \right) ^{\frac{1}{2}-\delta} N_1^{-\frac{2}{2+\delta}+\frac{d\delta}{2(2+\delta )}} \norm{P_{N_0} v_0}_{V^2_{\Delta}} \norm{P_{N_1} v}_{V^2_{\Delta}} \norm{\phi ^{\omega}}_{L^2_x} \\
& \hspace*{200pt} \times \prod _{k=2}^{m-1} \norm{\mathfrak{R}_k z}_{L^{\frac{2(m-2)(2+\delta )}{\delta}}_{t,x}} \\
& \lesssim
\sum _{\substack{N_0,N_1 \in 2^{\N _0} \\ N_1 \gtrsim N_0}}
N_0^{2s_c+\frac{1}{m-1}+\frac{1}{2}-\delta} N_1^{-s-s_c-\frac{3}{2}-\frac{(d+2)\delta}{2+\delta}+\delta} \norm{v}_{Y^{s_c}} R^{m-1} \\
& \lesssim
\norm{v}_{Y^{s_c}} R^{m-1}
\end{align*}
for $\omega \in E^m_R$.
Here, we have used the fact that $s> \frac{d}{2}-1$ and $\delta >0$ is sufficiently small in the last inequality.

\textbf{Subsubcase 3-4-2:} $l \ge 2$

Firstly, we consider the case $Q_0=Q_{>N_l^2}$.
By H\"{o}lder's inequality, the Sobolev embedding $H^{\frac{d\delta}{2(2+\delta )}} (\R ^d) \hookrightarrow L^{2+\delta}(\R ^d)$, Lemma \ref{lem:Q}, and Corollary \ref{cor:biStv}, we get
\begin{align*}
& I_{\substack{N_l \sim N_m \gtrsim N_0 \\ N_1 \le \dots \le N_l \\ N_{l+1} \le \dots \le N_m}} \\
& \lesssim
\sum _{\substack{N_0,N_1, \dots , N_m \in 2^{\N _0} \\ N_l \sim N_m \gtrsim N_0 \\ N_1 \le \dots \le N_l \\ N_{l+1} \le \dots \le N_m}}
N_0^{s_c+1} \norm{Q_{>N_l^2} P_{N_0} v_0}_{L^{2+\delta}_{t,x}} \left( \prod _{j=1}^{l-2} \norm{\mathfrak{R}_j v}_{L^{\infty}_{t,x}} \right) \norm{\mathfrak{R}_{l-1} v \mathfrak{R}_l v}_{L^2_{t,x}} \prod _{k=l+1}^m \norm{\mathfrak{R}_k z}_{L^{\frac{2(m-l)(2+\delta )}{\delta}}_{t,x}} \\
& \lesssim
\sum _{\substack{N_0,N_1, \dots , N_m \in 2^{\N _0} \\ N_l \sim N_m \gtrsim N_0 \\ N_1 \le \dots \le N_l \\ N_{l+1} \le \dots \le N_m}}
N_0^{s_c+1} N_l^{-\frac{2}{2+\delta}} N_0^{\frac{d\delta}{2(2+\delta )}} \left( \prod _{j=1}^{l-2} N_j^{\frac{d}{2}} \right) N_{l-1}^{\frac{d}{2}-1} \left( \frac{N_{l-1}}{N_l} \right) ^{\frac{1}{2}-\delta} \norm{P_{N_0} v_0}_{V^2_{\Delta}} \\
& \hspace*{100pt} \times \left( \prod _{j=1}^l \norm{P_{N_j} v}_{V^2_{\Delta}} \right) \prod _{k=l+1}^m \norm{\mathfrak{R}_k z}_{L^{\frac{2(m-l)(2+\delta )}{\delta}}_{t,x}} \\
& \lesssim
\sum _{\substack{N_0,N_1, \dots , N_l \in 2^{\N _0} \\ N_l \gtrsim N_0 \\ N_1 \le \dots \le N_l}}
N_0^{s_c+1+\frac{d\delta}{2(2+\delta )}} \left( \prod _{j=1}^{l-2} N_j^{\frac{1}{m-1}} \right) N_{l-1}^{\frac{1}{m-1}-\frac{1}{2}-\delta} N_l^{-s-s_c-\frac{1}{2}-\frac{2}{2+\delta}+\delta} \norm{v}_{Y^{s_c}}^l R^{m-l} \\
& \lesssim
\sum _{\substack{N_{l-1}, N_l \in 2^{\N _0} \\ N_l \ge N_{l-1}}}
N_{l-1}^{\frac{l-1}{m-1}-\frac{1}{2}-\delta} N_l^{-s-\frac{1}{2}+\frac{(d+2)\delta}{2+\delta}+\delta} \norm{v}_{Y^{s_c}}^l R^{m-l} \\
& \lesssim
\norm{v}_{Y^{s_c}}^l R^{m-l}
\end{align*}
for $\omega \in E^m_R$.
Here, we have used the fact that $s> \max (-\frac{1}{2},-\frac{1}{m-1})$ and $\delta >0$ is sufficiently small in the last inequality.

Secondly, we consider the case $Q_1=Q_{>N_l^2}$.
>From H\"{o}lder's inequality, the embeddings $H^{s_c}(\R ^d) \hookrightarrow L^{(m-1)d}(\R ^d)$, $H^{\frac{d\delta}{2(2+\delta )}} (\R ^d) \hookrightarrow L^{2+\delta}(\R ^d)$, Lemmas \ref{lem:Stv}, and \ref{lem:Q}, we have
\begin{align*}
& I_{\substack{N_l \sim N_m \gtrsim N_0 \\ N_1 \le \dots \le N_l \\ N_{l+1} \le \dots \le N_m}} \\
& \lesssim
\sum _{\substack{N_0,N_1, \dots , N_m \in 2^{\N _0} \\N_l \sim N_m \gtrsim N_0 \\ N_1 \le \dots \le N_l \\ N_{l+1} \le \dots \le N_m}}
N_0^{s_c+1} \norm{\mathfrak{R}_0 v_0}_{L^4_t L^{\frac{2d}{d-1}}} \norm{Q_{>N_0^2} P_{N_1} v}_{L^{2+\delta}_t L^{(m-1)d}_x} \left( \prod _{j=2}^{l-1} \norm{\mathfrak{R}_j v}_{L^{\infty}_t L^{(m-1)d}_x} \right) \\
& \hspace*{100pt} \times \norm{\mathfrak{R}_l v}_{L^4_t L^{\frac{2d}{d-1}}_x} \prod _{k=l+1}^m \norm{\mathfrak{R}_k z}_{L^{\frac{2(2+\delta)}{\delta}}_t L^{(m-1)d}_x} \\
& \lesssim
\sum _{\substack{N_0,N_1, \dots , N_m \in 2^{\N _0} \\N_l \sim N_m \gtrsim N_0 \\ N_1 \le \dots \le N_l \\ N_{l+1} \le \dots \le N_m}}
N_0^{s_c+1} N_0^{-\frac{2}{2+\delta}} \left( \prod _{j=1}^{l-1} N_j^{s_c} \right) \norm{P_{N_0} v_0}_{V^2_{\Delta}} \left( \prod _{j=1}^l \norm{P_{N_j} v}_{V^2_{\Delta}} \right) \prod _{k=l+1}^m \norm{\mathfrak{R}_k z}_{L^{\frac{2(2+\delta)}{\delta}}_t L^{(m-1)d}_x} \\
& \lesssim
\sum _{\substack{N_0,N_1, \dots , N_l \in 2^{\N _0} \\ N_l \gtrsim N_0 \\ N_1 \le \dots \le N_l}}
N_0^{s_c+\frac{\delta}{2+\delta}} N_l^{-s-s_c+\delta} \norm{v}_{Y^{s_c}}^l R^{m-l} \\
& \lesssim
\norm{v}_{Y^{s_c}}^l R^{m-l}
\end{align*}
for $\omega \in E^m_R$.
The case where $Q_j=Q_{>N_l^2}$ ($j=2, \dots ,l-1$) is similarly handled.

Thirdly, we consider the case $Q_l=Q_{>N_l^2}$.
>From H\"{o}lder's inequality, the embeddings $H^{s_c}(\R ^d) \hookrightarrow L^{(m-1)d}(\R ^d)$, $H^{\frac{d\delta}{2(2+\delta )}} (\R ^d) \hookrightarrow L^{2+\delta}(\R ^d)$, $H^{\frac{d}{2}-\frac{d}{d+\delta}}(\R ^d) \hookrightarrow L^{d+\delta} (\R ^d)$, Lemmas \ref{lem:Stv}, and \ref{lem:Q}, we have
\begin{align*}
& I_{\substack{N_l \sim N_m \gtrsim N_0 \\ N_1 \le \dots \le N_l \\ N_{l+1} \le \dots \le N_m}} \\
& \lesssim
\sum _{\substack{N_0,N_1, \dots , N_m \in 2^{\N _0} \\ N_l \sim N_m \gtrsim N_0 \\ N_1 \le \dots \le N_l \\ N_{l+1} \le \dots \le N_m}}
N_0^{s_c+1} \norm{\mathfrak{R}_0 v_0}_{L^4_t L^{\frac{2d}{d-1}}} \left( \prod _{j=1}^{l-2} \norm{\mathfrak{R}_j v}_{L^{\infty}_{t,x}} \right) \norm{\mathfrak{R}_{N_{l-1}} v}_{L^4_t L^{\frac{2d}{d-1}}_x} \norm{Q_{>N_l^2} P_{N_l} v}_{L^{2+\delta}_t L^{d+\delta}_x} \\
& \hspace*{100pt} \times \prod _{k=l+1}^m \norm{\mathfrak{R}_k z}_{L^{\frac{2(m-l)(2+\delta )}{\delta}}_t L^{\frac{d(m-l)(d+\delta)}{\delta}}_x} \\
& \lesssim
\sum _{\substack{N_0,N_1, \dots , N_m \in 2^{\N _0} \\ N_l \sim N_m \gtrsim N_0 \\ N_1 \le \dots \le N_l \\ N_{l+1} \le \dots \le N_m}}
N_0^{s_c+1} \left( \prod _{j=1}^{l-2} N_j^{\frac{d}{2}} \right) N_l^{-\frac{2}{2+\delta}+\frac{d}{2}-\frac{d}{d+\delta}} \norm{P_{N_0} v_0}_{V^2_{\Delta}} \left( \prod _{j=1}^l \norm{P_{N_j} v}_{V^2_{\Delta}} \right) \\
& \hspace*{100pt} \times \prod _{k=l+1}^m \norm{\mathfrak{R}_k z}_{L^{\frac{2(m-l)(2+\delta )}{\delta}}_t L^{\frac{d(m-l)(d+\delta)}{\delta}}_x} \\
& \lesssim
\sum _{\substack{N_0,N_1, \dots , N_l \in 2^{\N _0} \\ N_l \gtrsim N_0 \\ N_1 \le \dots \le N_l}}
N_0^{s_c+1} \left( \prod _{j=1}^{l-2} N_j^{\frac{1}{m-1}} \right) N_{l-1}^{-s_c} N_l^{-s-s_c-\frac{2}{2+\delta}+\frac{d}{2}-\frac{d}{d+\delta}} \norm{v}_{Y^{s_c}}^l R^{m-l} \\
& \lesssim
\sum _{\substack{N_{l-1},N_l \in 2^{\N _0} \\ N_{l-1} \le N_l}}
N_{l-1}^{-s_c+\frac{l-2}{m-1}} N_l^{-s+\frac{d}{2}-1+\frac{\delta}{2+\delta}+\frac{\delta}{d+\delta}} \norm{v}_{Y^{s_c}}^l R^{m-l} \\
& \lesssim
\norm{v}_{Y^{s_c}}^l R^{m-l}
\end{align*}
for $\omega \in E^m_R$.
Here, we have used the fact that $s> \max( \frac{d}{2}-1, - \frac{1}{m-1})$ and $\delta >0$ is sufficiently small in the last inequality.

\textbf{Subcase 3-5:} $N_{m-1} \sim N_m \gtrsim N_0, N_l$

We assume $l \le m-2$ otherwise it is reduced to the subcase 3-4.
Firstly, we consider the case $Q_0=Q_{>N_m^2}$.
>From H\"{o}lder's inequality, the embedding $H^{\frac{d\delta}{2(2+\delta )}} (\R ^d) \hookrightarrow L^{2+\delta}(\R ^d)$, Lemma \ref{lem:Q}, and Corollary \ref{cor:biStv}, we have
\begin{align*}
& I_{\substack{N_{m-1} \sim N_m \gtrsim N_0, N_l \\ N_1 \le \dots \le N_l \\ N_{l+1} \le \dots \le N_m}} \\
& \lesssim
\sum _{\substack{N_0,N_1, \dots , N_m \in 2^{\N _0} \\ N_{m-1} \sim N_m \gtrsim N_0, N_l \\ N_1 \le \dots \le N_l \\ N_{l+1} \le \dots \le N_m}}
N_0^{s_c+1} \norm{Q_{>N_m^2} P_{N_0} v_0}_{L^{2+\delta}_{t,x}} \left( \prod _{j=1}^{l-1} \norm{\mathfrak{R}_j v}_{L^{\infty}_{t,x}} \right) \norm{\mathfrak{R}_l v \mathfrak{R}_m z}_{L^2_{t,x}} \prod _{k=l+1}^{m-1} \norm{\mathfrak{R}_k z}_{L^{\frac{2(m-l-1)(2+\delta )}{\delta}}_{t,x
}} \\
& \lesssim
\sum _{\substack{N_0,N_1, \dots , N_m \in 2^{\N _0} \\ N_{m-1} \sim N_m \gtrsim N_0, N_l \\ N_1 \le \dots \le N_l \\ N_{l+1} \le \dots \le N_m}}
N_0^{s_c+1} N_m^{-\frac{2}{2+\delta}} N_0^{\frac{d\delta}{2(2+\delta )}} \left( \prod _{j=1}^{l-1} N_j^{\frac{d}{2}} \right) N_l^{\frac{d}{2}-1} \left( \frac{N_l}{N_m} \right) ^{\frac{1}{2}-\delta} \norm{P_{N_0} v_0}_{V^2_{\Delta}} \\
& \hspace*{100pt} \times \left( \prod _{j=1}^l \norm{P_{N_j} v}_{V^2_{\Delta}} \right) \norm{\phi ^{\omega}}_{L^2_x} \prod _{k=l+1}^{m-1} \norm{\mathfrak{R}_k z}_{L^{\frac{2(m-l-1)(2+\delta )}{\delta}}_{t,x
}} \\
& \lesssim
\sum _{\substack{N_0,N_1, \dots , N_m \in 2^{\N _0} \\ N_{m-1} \sim N_m \gtrsim N_0, N_l \\ N_1 \le \dots \le N_l \\ N_{l+1} \le \dots \le N_m}}
N_0^{s_c+1+\frac{d\delta}{2(2+\delta )}} \left( \prod _{j=1}^{l-1} N_j^{\frac{1}{m-1}} \right) N_l^{\frac{1}{m-1}-\frac{1}{2}-\delta} N_m^{-2s-\frac{1}{2}-\frac{2}{2+\delta}+\delta} \norm{v}_{Y^{s_c}}^l R^{m-l} \\
& \lesssim
\sum _{\substack{N_l,N_m \in 2^{\N _0} \\ N_m \gtrsim N_l}}
N_l^{\frac{l}{m-1}-\frac{1}{2}-\delta} N_m^{-2s+s_c-\frac{1}{2}+\frac{(d+2)\delta}{2+\delta}+\delta} \norm{v}_{Y^{s_c}}^l R^{m-l} \\
& \lesssim \norm{v}_{Y^{s_c}}^l R^{m-l}
\end{align*}
for $\omega \in E^m_R$.
Here, we have used the fact that $s> \frac{s_c}{2}$ and $\delta >0$ is sufficiently small in the last inequality.

Secondly, we consider the case $Q_1=Q_{>N_m^2}$.
>From H\"{o}lder's inequality, the embeddings $H^{s_c}(\R ^d) \hookrightarrow L^{(m-1)d}(\R ^d)$, $H^{\frac{d\delta}{2(2+\delta )}} (\R ^d) \hookrightarrow L^{2+\delta}(\R ^d)$, Lemmas \ref{lem:Stv}, and \ref{lem:Q}, we have
\begin{align*}
& I_{\substack{N_{m-1} \sim N_m \gtrsim N_0, N_l \\ N_1 \le \dots \le N_l \\ N_{l+1} \le \dots \le N_m}} \\
& \lesssim
\sum _{\substack{N_0,N_1, \dots , N_m \in 2^{\N _0} \\ N_{m-1} \sim N_m \gtrsim N_0, N_l \\ N_1 \le \dots \le N_l \\ N_{l+1} \le \dots \le N_m}}
N_0^{s_c+1} \norm{\mathfrak{R}_0 v_0}_{L^4_t L^{\frac{2d}{d-1}}} \norm{Q_{>N_m^2} P_{N_1} v}_{L^{2+\delta}_t L^{(m-1)d}_x} \left( \prod _{j=2}^l \norm{\mathfrak{R}_j v}_{L^{\infty}_t L^{(m-1)d}_x} \right) \\
& \hspace*{100pt} \times \left( \prod _{k=l+1}^{m-1} \norm{\mathfrak{R}_k z}_{L^{\frac{2(m-l-1)(2+\delta )}{\delta}}_t L^{(m-1)d}_x} \right) \norm{\mathfrak{R}_{N_m} z}_{L^4_t L^{\frac{2d}{d-1}}_x}  \\
& \lesssim
\sum _{\substack{N_0,N_1, \dots , N_m \in 2^{\N _0} \\ N_{m-1} \sim N_m \gtrsim N_0, N_l \\ N_1 \le \dots \le N_l \\ N_{l+1} \le \dots \le N_m}}
N_0^{s_c+1} N_m^{-\frac{2}{2+\delta}} \left( \prod _{j=1}^l N_j^{s_c} \right) \norm{P_{N_0} v_0}_{V^2_{\Delta}} \left( \prod _{j=1}^l \norm{P_{N_j} v}_{V^2_{\Delta}} \right) \\
& \hspace*{100pt} \times \left( \prod _{k=l+1}^{m-1} \norm{\mathfrak{R}_k z}_{L^{\frac{2(m-l-1)(2+\delta )}{\delta}}_t L^{(m-1)d}_x} \right) \norm{\mathfrak{R}_{N_m} z}_{L^4_t L^{\frac{2d}{d-1}}_x}  \\
& \lesssim
\sum _{\substack{N_0,N_m \in 2^{\N _0} \\ N_m \gtrsim N_0}}
N_0^{s_c+1} N_m^{-2s-\frac{2}{2+\delta}+\delta} \norm{v}_{Y^{s_c}}^l R^{m-l} \\
& \lesssim
\norm{v}_{Y^{s_c}}^l R^{m-l}
\end{align*}
for $\omega \in E^m_R$.
Here, we have used the fact that $s> \frac{s_c}{2}$ and $\delta >0$ is sufficiently small in the last inequality.
\end{proof}

\section{Proof of Main results} \label{S:proof_WP}

By a standard contraction argument, we deduce Theorems \ref{thm:LWP} and \ref{thm:GWP} from Lemmas \ref{lem:nonestl} and \ref{lem:nonest} respectively.
We give a rough outline (see \cite{BOP2} and \cite{HO15}).

\begin{proof}[Proof of Theorem \ref{thm:LWP}]
Let $\eta$ be small enough such that
\begin{equation} \label{cond:eta1}
2 C_1' \eta ^{m-1} \le 1 , \quad 2 C_2' \eta ^{m-1} \le \frac{1}{4},
\end{equation}
where $C_1'$ and $C_2'$ are the constants as in \eqref{eq:nonestl1} and \eqref{eq:nonestl2}.
For any $R>0$, we choose $T=T(R)$ such that
\[
T := \min \left( \frac{\eta}{2 C_1' R^{m}} , \frac{1}{4C_2' R^{m-1}} \right) ^{\frac{100}{\delta}} .
\]
Then, by Lemma \ref{lem:nonestl} and $Z^{s_c}_T \hookrightarrow Y^{s_c}_T$, the mapping $v \mapsto \Gamma [\mathcal{N}_{m} (v+z)]$ is a contraction on the ball $B_{\eta}$ defined by
\[
B_{\eta} := \{ u \in Z^{s_c}_T : \norm{u}_{Z^{s_c}_T} \le \eta \}
\]
outside a set of probability $ \le C \exp ( - c \frac{1}{T^{\gamma} \norm{\phi}_{H^s}})$ for some $\gamma > 0$, which leads the almost sure local in time well-posedness.
\end{proof}

\begin{proof}[Proof of Theorem \ref{thm:GWP}]

The same argument as in Corollary 3.4 in \cite{HHK} or Appendix C in \cite{IKO} yields that \eqref{eq:nonest1}, \eqref{eq:nonest2} with $T=\infty$ hold.
Let $\eta >0$ be sufficiently small such that
\begin{equation} \label{cond:eta2}
2C_1 \eta  ^{m-1} \le 1 ,\quad
3 C_2 \eta ^{m-1} \le \frac{1}{2} ,
\end{equation}
where $C_1$ and $C_2$ are the constants as in \eqref{eq:nonest1} and \eqref{eq:nonest2}.
Then, by Lemma \ref{lem:nonest} with $T=\infty$ and $Z^{s_c}_{\infty} \hookrightarrow Y^{s_c}_{\infty}$, the mapping $v \mapsto \Gamma [\mathcal{N}_{m} (v+z)]$ is a contraction on the ball $B_{R}$ defined by
\[
B_{\eta} := \{ u \in Z^{s_c}_{\infty} : \norm{u}_{Z^{s_c}_{\infty}} \le \eta \}
\]
outside a set of probability $\le C \exp ( - c \frac{\eta ^2}{\norm{\phi}_{H^s}})$.
We thus obtain the almost sure global in time well-posedness.

The scattering follows from \eqref{eq:nonest1} and the property of the $U^2$-space.
\end{proof}

\section*{Acknowledgment}

The work of the second author was partially supported by JSPS KAKENHI Grant number 26887017.

\end{document}